\documentclass[12pt]{amsart}
\usepackage{latexsym}
\usepackage{amssymb}
\usepackage{amsmath}
\usepackage[mathscr]{eucal}
\usepackage[pdftex]{hyperref}
\input xypic
\newtheorem{propo}{Proposition}[section]
\newtheorem{proposition}[propo]{Proposition}

\newtheorem{lemma}[propo]{Lemma}

\newtheorem{theorem}[propo]{Theorem}

\newcommand{\Ker}{\operatorname{Ker}}

\newcommand{\Irr}{{\mathrm {Irr}}}
\newcommand{\IBR}{{\mathrm {IBr}}}

\renewcommand{\Im}{{\mathrm {Im}}}

\newcommand{\Hom}{{\mathrm {Hom}}}

\newcommand{\Stab}{{\mathrm {Stab}}}

\newcommand{\Char}{{\mathrm {char}}}

\newcommand{\CC}{{\mathbb C}}

\newcommand{\FF}{{\mathbb F}}

\begin{document}
\title[rank $3$ permutation modules for $O_{2n}^{\pm}(2)$ and $U_{m}(2)$]
{The structure of rank $3$ permutation modules\\ for
$O_{2n}^{\pm}(2)$ and $U_{m}(2)$ acting on nonsingular points}

\author{Jonathan I. Hall}
\address{Department of Mathematics, Michigan State University, East Lansing,
MI 48824, USA} \email{jhall@math.msu.edu }
\author{Hung Ngoc Nguyen}
\address{Department of Mathematics, Michigan State University, East Lansing,
MI 48824, USA} \email{hungnguyen@math.msu.edu}
\date{\today}

\begin{abstract}
We study the odd-characteristic structure of permutation modules for
the rank~$3$ natural actions of~$O_{2n}^{\pm}(2)$ ($n\geq3$)
and~$U_{m}(2)$ ($m\geq4$) on nonsingular points of their standard
modules.
\end{abstract}

\maketitle


\section{Introduction}
Let~$G$ be either a symplectic, an orthogonal, or a unitary group.
The primitive rank $3$ permutation representations of $G$ have been
classified in~\cite{KL}. As a result, the natural action of $G$ on
the set of singular points (by points we mean $1$-dimensional
subspaces) of its standard module is always rank $3$ and the
associated permutation module has been studied in many papers
(see~\cite{LST,L1,L2,ST}). On the other hand, the action of~$G$ on
the set of nonsingular points is rank~$3$ if and only if~$G$ is
orthogonal groups $O_{2n}^{\pm}(2)$ with $n\geq3$ or unitary groups
$U_{m}(2)$ with $m\geq4$. The purpose of this paper is to describe
the odd-characteristic structure, including the composition factors
and submodule lattices, of the permutation modules for these groups
acting on nonsingular points.

We partly utilize the notations and methods from~\cite{L1}
and~\cite{ST}. Let~$\FF$ be an algebraically closed field of cross
characteristic~$\ell$ (i.e., $\ell$ is different from the
characteristic of the underlying field of $G$). If the action of~$G$
on a set $\Omega$ is rank~$3$ then the $\FF G$-module $\FF\Omega$
has two special submodules, which are so-called \emph{graph
submodules} in the terminology of Liebeck~\cite{L1}. Similar to the
study of cross-characteristic permutation modules on singular
points, these graph submodules in our problem are \emph{minimal} in
an appropriate sense (see Propositions~\ref{mainO+}, \ref{mainO-},
\ref{mainUeven}, and~\ref{mainUodd}). When the graph submodules are
different, their direct sum will be a submodule of codimension~$1$
in the permutation module and therefore the structure can be
determined without significant effort. When they are the same (that
is when $\Char(\FF)=3$, as we will see later on), the problem
becomes more complicated. We will look at both module and character
points of view to handle this case.

The complete description of submodule structures of the permutation
modules is given in the following theorem.

\begin{theorem} \label{theorem}Let $\FF$ be an algebraically closed field
of odd characteristic~$\ell$. Let~$G$ be either $O^{\pm}_{2n}(2)$
with $n\geq3$ or $U_{m}(2)$ with~$m\geq 4$ and $P$ be the set of
nonsingular points of its standard module. Then the $\FF
G$-permutation module $\FF P$ of $G$ acting naturally on~$P$ has the
structure as described in Tables~1,2,3, and~4, where the socle
series, submodule lattices, and dimensions of composition factors
are determined. In these tables, $\delta_{i,j}=1$ if $i\mid j$ and
$0$ otherwise. In table~\ref{tableO-} where~$G=O_{2n}^-(2)$,
$\omega$ is the nontrivial $\FF G$-module of dimension~$1$.
\end{theorem}

The paper is organized as follows. In the next section, we collect
some standard results of rank~$3$ permutation modules which will be
used frequently later on. We also establish some relations between
the actions of $G$ on singular points and nonsingular points
in~$\S2$. Each of the families of groups~$O^+_{2n}(2)$,
$O^-_{2n}(2)$, $U_{2n}(2)$, and~$U_{2n+1}(2)$ is treated
individually in~$\S 3$, $\S4$, and~$\S 5$, respectively. Proof of
the main theorem for each family is at the end of the corresponding
section. Since proofs are similar in many places, we only give
detailed arguments for the case $G=O^+_{2n}(2)$.

\medskip

\textbf{Notation}: Throughout the paper, $O^{\pm}_{2n}(2)$
and~$U_{m}(2)$ are full orthogonal and unitary groups, respectively.
Also, $SU_{m}(2)$ is the special unitary group. The commutator
subgroup of $O^{\pm}_{2n}(2)$ is denoted by $\Omega^\pm_{2n}(2)$.
If~$G$ is a group, $\Irr(G)$ (resp. $\IBR_{\ell}(G)$) will be the
set of irreducible complex (resp. $\ell$-Brauer) characters of~$G$.
We denote by $\overline{\chi}$ the reduction modulo~$\ell$ of a
complex character $\chi$. If $\varphi\in\IBR_{\ell}(G)$
and~$\lambda$ is a constituent of~$\varphi$ with multiplicity~$k$,
then we sometimes say that $\varphi$ has $k$ constituents~$\lambda$.

\footnotesize
\begin{table}[h]\label{tableO+}\caption{Submodule structure of $\FF O^+_{2n}(2)$-module $\FF P$.}
\begin{tabular}{ll} \hline
Conditions on $\ell$ and $n$\quad\quad\quad & Structure of $\FF
P$\\\hline
\xymatrix@C=8pt@R=8pt{\ell\neq2,3;\ell\nmid (2^n-1)}& \xymatrix@C=8pt@R=8pt{\FF\oplus X\oplus Y }\\

\xymatrix@C=8pt@R=8pt{\ell\neq2,3;\ell\mid
(2^n-1)}&\xymatrix@C=8pt@R=8pt{
&&\FF\ar@{-}[d]\\
X&\oplus&Y\ar@{-}[d]\\
&&\FF}\\

\xymatrix@C=8pt@R=8pt{\ell=3;n \text{ even }}&\xymatrix@C=8pt@R=8pt{
\FF\ar@{-}[dr]&&X\ar@{-}[dl]\\
&Z\ar@{-}[dl]\ar@{-}[dr]&\\
\FF&&X}\\

\xymatrix@C=8pt@R=8pt{\ell=3;n \text{ odd}}  &
\xymatrix@C=8pt@R=8pt{
&&X\ar@{-}[d]\\
\FF&\oplus&Z\ar@{-}[d]\\
&&X}\\\hline
\end{tabular}

where, $\dim X=\frac{(2^n-1)(2^{n-1}-1)}{3}$, $\dim
Y=\frac{2^{2n}-4}{3}-\delta_{\ell,2^n-1}$, \\and $\dim
Z=\frac{(2^n-1)(2^{n-1}+2)}{3}-1-\delta_{\ell,2^n-1}$.\normalsize
\end{table}

\begin{table}[h]\caption{Submodule structure of $\FF O^-_{2n}(2)$-module $\FF P$.}\label{tableO-}
\begin{tabular}{ll} \hline
Conditions on $\ell$ and $n$\quad\quad\quad & Structure of $\FF
P$\\\hline
\xymatrix@C=8pt@R=8pt{\ell\neq2,3;\ell\nmid (2^n+1)}& \xymatrix@C=8pt@R=8pt{\FF\oplus X\oplus Y }\\

\xymatrix@C=8pt@R=8pt{\ell\neq2,3;\ell\mid
(2^n+1)}&\xymatrix@C=8pt@R=8pt{
&&\FF\ar@{-}[d]\\
X&\oplus&Y\ar@{-}[d]\\
&&\FF}\\

\xymatrix@C=8pt@R=8pt{\ell=3;n \text{ even}}  &
\xymatrix@C=8pt@R=8pt{
&&&X\ar@{-}[dl]\ar@{-}[dr]&\\
\FF&\oplus&\omega\ar@{-}[dr]&&Z\ar@{-}[dl]\\
&&&X&}\\

\xymatrix@C=8pt@R=8pt{\ell=3;n \text{ odd }}&\xymatrix@C=8pt@R=8pt{
\FF\ar@{-}[d]&&X\ar@{-}[dll]\ar@{-}[d]\\
Z\ar@{-}[d]\ar@{-}[drr]&&\omega\ar@{-}[d]\\
\FF&&X}\\\hline
\end{tabular}

where, $\dim X=\frac{(2^n+1)(2^{n-1}+1)}{3}-\delta_{3,\ell}$, $\dim
Y=\frac{2^{2n}-4}{3}-\delta_{\ell,2^n-1}$, \\and $\dim
Z=\frac{(2^n+1)(2^{n-1}-2)}{3}-1+\delta_{\ell,2^n-1}$.\normalsize
\end{table}

\begin{table}[h]\caption{Submodule structure of $\FF U_{2n}(2)$-module $\FF P$.}\label{tableUeven}
\begin{tabular}{ll} \hline
Conditions on $\ell$ and $n$\quad\quad\quad & Structure of $\FF
P$\\\hline
\xymatrix@C=8pt@R=8pt{\ell\neq2,3;\ell\nmid (2^{2n}-1)}& \xymatrix@C=8pt@R=8pt{\FF\oplus X\oplus Y }\\

\xymatrix@C=8pt@R=8pt{\ell\neq2,3;\ell\mid
(2^{2n}-1)}&\xymatrix@C=8pt@R=8pt{
&&\FF\ar@{-}[d]\\
X&\oplus&Y\ar@{-}[d]\\
&&\FF}\\

\xymatrix@C=8pt@R=8pt{\ell=3;3\mid n}&\xymatrix@C=8pt@R=8pt{
\FF\ar@{-}[d]&&Z\ar@{-}[dll]\ar@{-}[d]\\
W_2\ar@{-}[d]\ar@{-}[drr]&&W_1\ar@{-}[d]\\
\FF&&Z}\\

\xymatrix@C=8pt@R=8pt{\ell=3;3\nmid n} & \xymatrix@C=8pt@R=8pt{
&&&Z\ar@{-}[dl]\ar@{-}[dr]&\\
\FF&\oplus&W_1\ar@{-}[dr]&&W_2\ar@{-}[dl]\\
&&&Z&}\\ \hline
\end{tabular}

where, $\dim X=\frac{(2^{2n}-1)(2^{2n-1}+1)}{9}$, $\dim
Y=\frac{(2^{2n}+2)(2^{2n}-4)}{9}-\delta_{\ell,2^{2n}-1}$, $\dim
W_1=\frac{2^{2n}-1}{3}$, \\$\dim
W_2=\frac{(2^{2n}-1)(2^{2n-1}+1)}{9}-1-\delta_{3,n}$, and $\dim
Z=\frac{(2^{2n}-1)(2^{2n-1}-2)}{9}$.\normalsize
\end{table}

\begin{table}[h]\caption{Submodule structure of $\FF U_{2n+1}(2)$-module $\FF P$.}\label{tableUodd}
\begin{tabular}{ll} \hline
Conditions on $\ell$ and $n$\quad\quad\quad & Structure of $\FF
P$\\\hline
\xymatrix@C=8pt@R=8pt{\ell\neq2,3;\ell\nmid (2^{2n+1}+1)}& \xymatrix@C=8pt@R=8pt{\FF\oplus X\oplus Y }\\

\xymatrix@C=8pt@R=8pt{\ell\neq2,3;\ell\mid
(2^{2n+1}+1)}&\xymatrix@C=8pt@R=8pt{
&&\FF\ar@{-}[d]\\
X&\oplus&Y\ar@{-}[d]\\
&&\FF}\\

\xymatrix@C=8pt@R=8pt{\ell=3;3\mid n}& \xymatrix@C=8pt@R=8pt{
&&X\ar@{-}[d]\\
&&Z\ar@{-}[d]\\
&&\FF\ar@{-}[d]\\
\FF&\oplus & W\ar@{-}[d]\\
&&\FF\ar@{-}[d]\\
&&Z\ar@{-}[d]\\
&&X}\\

\xymatrix@C=8pt@R=8pt{\ell=3;n \equiv 1 \pmod{3}}  &
\xymatrix@C=8pt@R=8pt{
\FF\ar@{-}[dddr]&&X\ar@{-}[dl]\\
&Z\ar@{-}[dddl]\ar@{-}[d]&\\
&W\ar@{-}[d]&\\
&Z\ar@{-}[dr]&\\
\FF&&X}\\

\xymatrix@C=8pt@R=8pt{\ell=3;n\equiv 2 \pmod{3}}  &
\xymatrix@C=8pt@R=8pt{&&& X\ar@{-}[d]&\\
&&& Z\ar@{-}[dl]\ar@{-}[dr]&\\
\FF&\oplus & \FF\ar@{-}[dr] &  & W\ar@{-}[dl]\\
&&& Z \ar@{-}[d]&\\
&&& X&}\\ \hline
\end{tabular}

where, $\dim X=\frac{(2^{2n+1}+1)(2^{2n}-1)}{9}$, $\dim
Y=\frac{(2^{2n+1}-2)(2^{2n+1}+4)}{9}-\delta_{\ell,2^{2n+1}+1}$,\\
$\dim Z=\frac{2^{2n+1}-2}{3}$, and $\dim
W=\frac{(2^{2n+1}+1)(2^{2n}-4)}{9}-\delta_{3,n}$.\normalsize
\end{table}
\normalsize


\section{Preliminaries on rank $3$ permutation modules}
We start this section by recalling some results on rank~$3$
permutation modules from~\cite{Hi} and~\cite{L1}.

Let $G$ be a permutation group of rank~$3$ acting on the set
$\Omega$. Then, for each~$\alpha\in \Omega$, $G_\alpha$, the
stabilizer of~$\alpha$, acts on~$\Omega$ with $3$ orbits
$\{\alpha\}, \Delta(\alpha)$, and~$\Phi(\alpha)$. We choose the
notation so that $\Delta(\alpha)g=\Delta(\alpha g)$
and~$\Phi(\alpha)g=\Phi(\alpha g)$. Define the following parameters
associated with the action of~$G$ on~$\Omega$:
$$a=|\Delta(\alpha)|, b=|\Phi(\alpha)|,$$
$$r=|\Delta(\alpha)\cap\Delta(\beta)| \text{ for } \beta\in \Delta(\alpha),$$
$$s=|\Delta(\alpha)\cap\Delta(\gamma)| \text{ for } \gamma\in \Phi(\alpha).$$
These parameters do not depend on the choices of~$\alpha, \beta$,
and~$\gamma$.

Let $\mathbb{F}$ be a field of characteristic~$\ell$ and~$\FF
\Omega$ the associated permutation $\FF G$-module. For any subset
$\Delta$ of $\Omega$, we denote by $[\Delta]$ the element
$\Sigma_{\delta\in\Delta}\delta$ of $\FF \Omega$. Set $S(\FF
\Omega)=\{\sum_{\omega\in\Omega}a_\omega\omega\mid a_\omega\in\FF,
\sum a_\omega=0\}$ and $T(\FF \Omega)=\{c[\Omega]\mid c\in\FF\}$.
Note that $S(\FF\Omega)$ and $T(\FF\Omega)$ are $\FF G$-submodules
of $\FF \Omega$ of dimensions~$|\Omega|-1,1$, respectively.
Moreover, $T(\FF \Omega)$ is isomorphic to the one-dimensional
trivial module. We define a natural inner product on~$\FF\Omega$ by
$$\big\langle
\sum_{\omega\in\Omega}a_{\omega}\omega,\sum_{\omega\in\Omega}b_{\omega}\omega\big\rangle=\sum_{\omega\in\Omega}a_{\omega}b_\omega.$$
It is easy to see that $\langle.,.\rangle$ is non-singular and
$G$-invariant. In particular, $\FF \Omega$ is a self-dual $\FF
G$-module. If $U$ is a submodule of $\FF \Omega$, we denote by
$U^\perp$ the submodule of~$\FF \Omega$ consisting of all elements
orthogonal to~$U$.

For any element $c\in \FF$, let $U_c$ be the $\FF G$-submodule of
$\FF \Omega$ generated by all elements of the form
$v_{c,\alpha}=c\alpha+[\Delta(\alpha)], \alpha\in \Omega$ and $U'_c$
be the $\FF G$-submodule of $U_c$ generated by all elements
$v_{c,\alpha}-v_{c,\beta}=c(\alpha-\beta)+[\Delta(\alpha)]-[\Delta(\beta)],
\alpha, \beta\in\Omega$. It is obvious that $U'_c$ is always
contained in~$S(\FF\Omega)$. The following lemma tells us that
$U'_c=S(\FF\Omega)$ for most of~$c$.

\begin{lemma}[\cite{L1}] \label{liebeck}If $c$ is not a root of the quadratic
equation
\begin{equation}\label{quadratic} x^2+(r-s)x+(s-a)=0,
\end{equation}
then $U'_c=S(\FF\Omega)$. Moreover, if $c$ and $d$ are roots of this
equation then $\langle v_{c,\alpha},v_{d,\beta}\rangle=s$ for any
$\alpha, \beta\in \Omega$. Consequently, $\langle U'_c,
U_d\rangle=\langle U'_d, U_c\rangle=0$.
\end{lemma}

Define a linear transformation on~$\FF \Omega$ as follows
\begin{equation}\label{T}\begin{array}{cccc}T:&\FF \Omega &\rightarrow &\FF \Omega\\
&\alpha&\mapsto &[\Delta(\alpha)].\end{array}\end{equation} It is
easy to see that $T$ is an $\FF G$-homomorphism. Let $c,d$ be the
two roots of equation~(\ref{quadratic}). We have
$$T(v_{c,\alpha})=c[\Delta(\alpha)]+\sum_{\delta\in\Delta(\alpha)}
[\Delta(\delta)]=c[\Delta(\alpha)]+a\alpha+r[\Delta(\alpha)]+s[\Phi(\alpha)].$$
Therefore,
$$\begin{array}{ll}T(v_{c,\alpha}-v_{c,\beta})&=(a-s)(\alpha-\beta)+(r-s+c)([\Delta(\alpha)]-[\Delta(\beta)])\\
&=-cd(\alpha-\beta)-d([\Delta(\alpha)]-[\Delta(\beta)])=-d(v_{c,\alpha}-v_{c,\beta}).\end{array}$$
Thus, for any $v\in U'_c$, $T(v)=-dv$. Similarly, for any $v\in
U'_d$, $T(v)=-cv$. $U'_c$ and $U'_d$ are called graph submodules of
the permutation module $\FF \Omega$.

Now we study more details about the permutation modules for
$O_{2n}^{\pm}(2)$ or $U_{m}(2)$ acting on nonsingular points.
Let~$G$ be either $O_{2n}^{\pm}(2)$ or $U_{m}(2)$. Let $P$ and $P^0$
be the sets of nonsingular and singular points, respectively, of the
standard module associated with $G$. Define
\begin{equation}\label{Q}\begin{array}{cccc}Q:&\FF P &\rightarrow &\FF P^0\\
&\alpha&\mapsto &[\Lambda(\alpha)]\end{array}\end{equation} and
\begin{equation}\label{R}\begin{array}{cccc}R:&\FF P^0 &\rightarrow &\FF P\\
&\alpha&\mapsto &[\Gamma(\alpha)],\end{array}\end{equation} where
$\Lambda(\alpha)$ is the set of all singular points orthogonal to
$\alpha\in P$ and $\Gamma(\alpha)$ is the set of all nonsingular
points orthogonal to $\alpha\in P^0$. It is clear that $Q$ and~$R$
are $\FF G$-homomorphisms. Moreover, $\Im(Q|_{S(\FF P)})\neq 0,
T(\FF P^0)$ and $\Im(R|_{S(\FF P^0)})\neq 0, T(\FF P)$. We have
proved the following lemma which is a very important connection
between structures of $\FF P$ and $\FF P^0$.

\begin{lemma}\label{main1} There exists a nonzero submodule of $\FF P^0$
which  is not $T(\FF P^0)$ so that it is isomorphic to a quotient of
$\FF P$. Similarly, there exists a nonzero submodule of $\FF P$
which is not $T(\FF P)$ so that it is isomorphic to a quotient
of~$\FF P^0$. Moreover, these statements are still true if $\FF P$
and $\FF P^0$ are replaced by $S(\FF P)$ and $S(\FF P^0)$,
respectively.
\end{lemma}

Let $\rho$ and $\rho^0$ be complex permutation characters of $G$
acting on $\FF P$ and $\FF P^0$, respectively. Since these actions
are rank~$3$, it is well-known that both $\rho$ and $\rho^0$ have
$3$ constituents, all of multiplicity~$1$ and exactly one of them is
the trivial character. This and Lemma~\ref{main1} imply the
following:

\begin{lemma}\label{main2} $\rho$ and $\rho^0$ have a common constituent which is not
trivial.
\end{lemma}

Finally, we record here a basic property of self-dual modules over a
group algebra.

\begin{lemma} \label{main3} If $U$ is a self-dual $\FF G$-module
having a self-dual, simple socle $X$, then both the head of $U$ and
the top layer of the socle series of $U$ are isomorphic to~$X$.
\end{lemma}

\begin{proof} This follows from Lemmas~8.2 and~8.4 of~\cite{La}.
\end{proof}


\section{The orthogonal groups $O^+_{2n}(2)$}\label{sectionO+} Let $V$ be a vector space of dimension
$2n\geq6$ over the field of $2$ elements $\mathbb{F}_2=\{0,1\}$. Let
$Q(\cdot)$ be a quadratic form on $V$ of type $+$ and
$(\cdot,\cdot)$ be the non-degenerate symmetric bilinear form on $V$
associated with $Q$ so that $Q(au+bv)=a^2Q(u)+b^2Q(v)+ab(u,v)$ for
any $a,b\in \FF_2,u,v\in V$. Then $G=O^+_{2n}(2)$ is the orthogonal
group of linear transformations of $V$ preserving $Q$.

We choose a basic of $V$ consisting of vectors
$\{e_1,...,e_n,f_1,...,f_n\}$ so that $Q(e_i)=Q(f_j)=0$,
$(e_i,e_j)=(f_i,f_j)=0$, and $(e_i,f_j)=\delta_{ij}$ for all
$i,j=1,...,n$. Let $P$ be the set of all nonsingular points in $V$.
Then $P=\{\langle\sum_1^n(a_ie_i+b_if_i)\rangle\mid
\sum_1^na_ib_i=1\}$ and $|P|=2^{2n-1}-2^{n-1}$.

Let $\alpha=e_1+f_1\in P$. Note that if $\delta$ is orthogonal to
$\alpha$ and $\epsilon$ is not, then $\delta$ and $\epsilon$ are in
different orbits of the action of $G_\alpha$ on $P$. With no loss,
we assume that $\Delta(\alpha)$ consists of elements of $P$ which
are not orthogonal to $\alpha$ and $\Phi(\alpha)$ consists of
elements in $P\setminus\{\alpha\}$ which are orthogonal to $\alpha$.
We have
$$\begin{array}{rl}\Delta(\alpha)&=\{\langle v\rangle\in P\mid (e_1+f_1,v)=1\}\\
&=\{\langle\sum_1^n(a_ie_i+b_if_i)\rangle\in P\mid a_1+b_1=1\}\\
&=\{\langle e_1+v_1\rangle,\langle f_1+v_1\rangle\mid \langle
v_1\rangle \text{ is a nonsingular point in } V_1\},\end{array}$$
where $V_1=\langle e_2,...,e_n,f_2,...,f_n\rangle$. So
$a=|\Delta(\alpha)|=2^{2n-2}-2^{n-1}$ and therefore,
$b=|\Phi(\alpha)|=2^{2n-2}-1$.

Let $\beta=\langle e_1+e_2+f_2\rangle\in\Delta(\alpha)$. Then
$$\begin{array}{cl}\Delta(\alpha)\cap\Delta(\beta)&=\{\langle v\rangle\in P\mid
(e_1+f_1,v)=(e_1+e_2+f_2,v)=1\}\\
&=\{\langle\sum_1^n(a_ie_i+b_if_i)\rangle\in P\mid
a_1+b_1=b_1+a_2+b_2=1\}\end{array}.$$ An easy calculation shows that
$$r=|\Delta(\alpha)\cap\Delta(\beta)|=2^{2n-3}-2^{n-2}.$$
Similarly, if $\gamma=\langle e_2+f_2\rangle\in\Phi(\alpha)$ then we
have
$$\Delta(\alpha)\cap\Delta(\gamma)=\{\langle v\rangle\in P\mid a_1+b_1=a_2+b_2=1\}$$
and
$$s=|\Delta(\alpha)\cap\Delta(\gamma)|=2^{2n-3}-2^{n-1}.$$
Now equation~(\ref{quadratic}) becomes
$$x^2+2^{n-2}x-2^{2n-3}=0,$$
which has two roots $2^{n-2}$ and $-2^{n-1}$. By
Lemma~\ref{liebeck}, $U'_c=S(\FF P)$ for any $c\neq 2^{n-2},
-2^{n-1}$. As in the study of permutation modules for finite
classical group acting on singular points (see~\cite{L1,L2}), the
graph submodules $U'_{2^{n-2}}$ and $U'_{-2^{n-1}}$ of $\FF P$ are
minimal when $\Char(\FF)\neq 2$ in the following sense.

\begin{proposition}\label{mainO+} Suppose that the characteristic of $\mathbb{F}$ is
odd. Then every nonzero $\FF G$-submodule of $\FF P$ either is
$T(\FF P)$ or contains a graph submodule, which is $U'_{2^{n-2}}$ or
$U'_{-2^{n-1}}$.
\end{proposition}

We use some ideas from the proof of a similar result for the
permutation module of $G$ acting on singular points (see~\cite{L1}),
but the proof presented here is much simpler. Define
$$\Delta_1=\{\langle\sum_{i=1}^n(a_ie_i+b_if_i)\rangle\in P\mid b_1=1, a_2+b_2=1\}$$
and
$$\Delta_2=\{\langle\sum_{i=1}^n(a_ie_i+b_if_i)\rangle\in P\mid b_1=1, a_2+b_2=0\}.$$
Putting $\Delta=\Delta_1\cup\Delta_2$ and $\Phi=P\backslash \Delta$.
We easily see that $$[\Delta(\langle
e_2+f_2\rangle)]-[\Delta(\langle
e_1+e_2+f_2\rangle)]=[\Delta_1]-[\Delta_2]$$ and
$$\Phi=\{\langle\sum_{i=1}^n(a_ie_i+b_if_i)\rangle\in P\mid b_1=0\}.$$

Consider a subgroup $H<G$ consisting of orthogonal transformations
sending elements of the basis $\{e_1$, $f_1$, $e_2$,
$f_2,...,e_n,f_n\}$ to the those of the basis
$\{e_1,f_1+\sum_{i=1}^na_ie_i+\sum_{i=2}^nb_if_i,e_2-b_2e_1,f_2-a_2e_1,...,e_n-b_ne_1,f_n-a_ne_1\}$
respectively, where $a_i,b_i\in\FF_2$ and $a_1=\sum_{i=2}^na_ib_i$.
Let $K$ be the subgroup of $H$ consisting of transformations fixing
$e_2+f_2$. Let $P_1$ be the set of nonsingular points in
$V_1=\langle e_2,f_2,...,e_n,f_n\rangle$. For each $\langle
w\rangle\in P_1$, we define $B_w=\{\langle w\rangle, \langle
e_1+w\rangle\}$. Note that $H$ is the stabilizer of the subspace
series $0=V_0\leq V_1\leq V_2\leq V_3=V $ where $V_1=\langle
e_1\rangle$, $V_2=\langle e_1\rangle^\perp$, and $[V_i,H]\leq
V_{i-1}$. An element $h\in H$ is uniquely determined by the image
$f_1h=f_1+\sum_{i=1}^na_ie_i+\sum_{i=2}^nb_if_i$. In particular, $H$
is regular on $\Delta$, and the lemmas follow directly.

\begin{lemma}\label{delta} The following holds:
\begin{enumerate}
\item[(i)] $|H|=2^{2n-2}$, $|K|=2^{2n-3}$, $|\Delta|=2^{2n-2}$, and $|\Delta_1|=|\Delta_2|=2^{2n-3}$;
\item[(ii)] $H$ acts transitively on $\Delta$ and $K$ has $2$ orbits
$\Delta_1,\Delta_2$ on $\Delta$.
\end{enumerate}
\end{lemma}

\begin{lemma}\label{phi} The following holds:
\begin{enumerate}
\item[(i)] $\Phi=\bigcup_{\langle w\rangle\in P_1}B_w$;
\item[(ii)] $K$ fixes $B_{e_2+f_2}$ point-wise and is transitive
on $B_w$ for every $\langle e_2+f_2\rangle\neq\langle w\rangle\in P_1$;
\item[(iii)] $H$ acts transitively on $B_w$ for every $\langle w\rangle\in P_1$.
\end{enumerate}
\end{lemma}

{\textbf{Proof of Proposition \ref{mainO+}}}. Suppose that $U$ is a
nonzero submodule of $\FF P$. We assume that $U$ is not $T(\FF P)$.
Then $U$ contains an element of the form
$$u=a\phi_1+b\phi_2+\sum_{\delta\in P\backslash\{\phi_1,\phi_2\}}a_\delta\delta,$$
where $a,b,a_\delta\in\FF$, $\phi_1\neq\phi_2\in P$, and $a\neq b$.
If $(\phi_1,\phi_2)=1$, we choose an element $\phi_3\in P$ so that
$(\phi_1,\phi_3)=(\phi_2,\phi_3)=0$. Since $a\neq b$, the
coefficient of $\phi_3$ in $u$ is different from either $a$ or $b$.
Therefore, with no loss, we can assume $(\phi_1,\phi_2)=0$.

Since $(e_2+f_2,e_1+e_2+f_2)=0$, there exists $g'\in G$ such that
$\phi_1g'=\langle e_2+f_2\rangle$ and $\phi_2g'=\langle
e_1+e_2+f_2\rangle$. Therefore, we can assume that
$u=a\phi_1+b\phi_2+\sum_{\delta\in
P\backslash\{\phi_1,\phi_2\}}a_\delta\delta\in U$ with
$\phi_1=\langle e_2+f_2\rangle$ and $\phi_2=\langle
e_1+e_2+f_2\rangle$.

Take an element $g\in G$ such that $e_1g=e_1$ and
$(e_2+f_2)g=e_1+e_2+f_2$. Then $\phi_1 g=\phi_2$ and $\phi_2
g=\phi_1$. So we have
$$u-ug=(a-b)\phi_1-(a-b)\phi_2+
\sum_{\delta\in P\backslash\{\phi_1,\phi_2\}}b_\delta\delta\in U,$$
where $b_\delta\in\FF$. Note that $u-ug\in S(\FF P)$. Therefore, if
$c_\delta=b_\delta/(a-b)$, we get $$
u_1:=(u-ug)/(a-b)=\phi_1-\phi_2+\sum_{\delta\in
P\backslash\{\phi_1,\phi_2\}}c_\delta\delta\in U\cap S(\FF P).
$$
Hence we have $u_2:=\sum_{k\in K}u_1k\in U\cap S(\FF P)$. Moreover,
by Lemmas~\ref{delta} and~\ref{phi},
$$u_2=2^{2n-3}(\phi_1-\phi_2)+\sum_{\delta\in\Delta}d_\delta\delta+\sum_{\langle
w\rangle\in P_1, \langle w\rangle\neq \phi_1}d_w[B_w],$$ where
$d_\delta,d_w\in\FF$. Therefore $u_3:=\sum_{h\in H}u_2h\in U\cap
S(\FF P)$ with
$$u_3=(\sum_{\delta\in\Delta}d_\delta)[\Delta]+2^{2n-2}\sum_{\langle w\rangle\in P_1,
\langle w\rangle\neq \phi_1}d_w[B_w]$$ again by Lemmas~\ref{delta}
and~\ref{phi}. It follows that
$$u_4:=2^{2n-2}u_2-u_3=2^{4n-5}(\phi_1-\phi_2)+\sum_{\delta\in\Delta}f_\delta\delta\in
U\cap S(\FF P),$$ where
$f_\delta=2^{2n-2}d_\delta-\sum_{\delta\in\Delta}d_\delta$. Hence
$$u_5:=\sum_{k\in
K}u_4k=2^{6n-8}(\phi_1-\phi_2)+f[\Delta_1]+f'[\Delta_2]\in U\cap
S(\FF P),$$ where $f=\sum_{\delta\in\Delta_1}f_\delta$ and
$f'=\sum_{\delta\in\Delta_2}f_\delta$. Since $u_5\in S(\FF P)$,
$f+f'=0$. Note that
$[\Delta_1]-[\Delta_2]=[\Delta(\phi_1)]-[\Delta(\phi_2)]$.
Therefore,
\begin{equation}\label{v_4}
u_5=2^{6n-8}(\phi_1-\phi_2)+f([\Delta(\phi_1)]-[\Delta(\phi_2)])\in
U.
\end{equation}

{\bf Case 1:} If $f=0$ then $u_5=2^{6n-8}(\phi_1-\phi_2)\in U$. From
the hypothesis that the characteristic of $\FF$ is odd, we have
$2^{6n-8}\neq0$. It follows that $\phi_1-\phi_2\in U$. Therefore
$\alpha-\beta\in U$ for every $\alpha,\beta\in P$. In other words,
$U\supseteq S(\FF P)$, which implies that $U$ contains both
$U'_{-2^{n-1}}$ and $U'_{2^{n-2}}$, as wanted.

{\bf Case 2:} If $f\neq 0$ then (\ref{v_4}) implies that
$(2^{6n-8}/f)(\phi_1-\phi_2)+[\Delta(\phi_1)]-[\Delta(\phi_2)]\in
U$. Therefore,
$(2^{6n-8}/f)(\alpha-\beta)+[\Delta(\alpha)]-[\Delta(\beta)]\in U$
for every $\alpha,\beta\in P$. In other words, $U\supseteq
U'_{2^{6n-8}/f}$. This and Lemma~\ref{liebeck} imply that $U$
contains either $U'_{-2^{n-1}}$ or $U'_{2^{n-2}}$. The proposition
is completely proved. \hfill$\Box$

\begin{proposition}\label{dimensionO+} If $\ell=\Char(\FF)\neq 2,3$, then
$$\dim U'_{2^{n-2}}=\frac{(2^n-1)(2^{n-1}-1)}{3} \text{ and } \dim U'_{-2^{n-1}}=\frac{2^{2n}-4}{3}.$$
\end{proposition}

\begin{proof} We have $v_{2^{n-2},\alpha}-v_{-2^{n-1},\alpha}=3\cdot
2^{n-2}\alpha$ for any $\alpha\in P$. It follows that, for any
$\alpha,\beta\in P$,
$$(v_{2^{n-2},\alpha}-v_{2^{n-2},\beta})-(v_{-2^{n-1},\alpha}-v_{-2^{n-1},\beta})=3\cdot2^{n-2}(\alpha-\beta).$$
Since $\ell\neq2,3$, $U'_{2^{n-2}}+U'_{-2^{n-1}}=S(\FF P)$. Note
that $T(v)=-2^{n-1}v$ for any $v\in U'_{2^{n-2}}$ and
$T(v)=2^{n-2}v$ for any $v\in U'_{-2^{n-1}}$. Therefore
$U'_{2^{n-2}}\cap U'_{-2^{n-1}}=\{0\}$ since $2^{n-2}\neq-2^{n-1}$.
So we have $U'_{2^{n-2}}\oplus U'_{-2^{n-1}}=S(\FF P)$. In
particular,
\begin{equation}\label{dimension1}\dim U'_{2^{n-2}}+\dim U'_{-2^{n-1}}=|P|-1=2^{2n-1}-2^{n-1}-1.
\end{equation}

The transposition $T$ (see (\ref{T})) has trace $0$ and the
eigenvalue $a=2^{2n-2}-2^{n-1}$ of multiplicity one. Since $T$ also
has two different eigenvalues $-2^{n-2}$ and $2^{n-1}$ with
corresponding eigenvector spaces $U'_{-2^{n-1}}$ and $U'_{2^{n-2}}$,
we have
\begin{equation}\label{dimension2}
-2^{n-2}\dim U'_{-2^{n-1}}+2^{n-1}\dim
U'_{2^{n-2}}=-(2^{2n-2}-2^{n-1}).
\end{equation}

Now (\ref{dimension1}) and (\ref{dimension2}) imply the proposition.
\end{proof}

As we mentioned in the introduction, the structure of $\FF P$ is
more complicated when $2^{n-2}=-2^{n-1}$ or equivalently when
$\ell=3$. We now study the decomposition of $\overline{\rho}$ (when
$\ell=3$) into irreducible characters, which are also called
\emph{constituents}.

\begin{lemma}\label{lemmaO+1} Suppose $\ell=3$. Then
\begin{enumerate}
\item[(i)] when $n$ is even, $\overline{\rho}$ has exactly $2$ trivial
constituents, $2$ constituents of degree $(2^n-1)(2^{n-1}-1)/3$, and
$1$ constituent of degree $(2^n-1)(2^{n-1}+2)/3-2$;
\item[(ii)] when $n$ is odd, $\overline{\rho}$ has exactly $1$ trivial
constituent, $2$ constituents of degree $(2^n-1)(2^{n-1}-1)/3$, and
$1$ constituent of degree $(2^n-1)(2^{n-1}+2)/3-1$.
\end{enumerate}
\end{lemma}

\begin{proof} The case $n=3$ can be checked directly using~\cite{Atl2}. Therefore we assume $n\geq4$.
From the proof of Proposition~\ref{dimensionO+}, in characteristic
$0$ we have $\FF P=T(\FF P)\oplus U'_{2^{n-2}}\oplus U'_{-2^{n-1}}$.
Let $\varphi$ and $\psi$ be the irreducible complex characters of
$G$ afforded by $U'_{2^{n-2}}$ and $U'_{-2^{n-1}}$, respectively.
Then we have $\rho=1+\varphi+\psi$, where
$\varphi(1)=(2^n-1)(2^{n-1}-1)/3$ and $\psi(1)=(2^{2n}-4)/3$. It has
been shown in~\cite{Ho} that the smallest degree of a nonlinear
irreducible characters of $G$ in cross-characteristic is at least
$(2^n-1)(2^{n-1}-1)/3$. Therefore $\overline{\varphi}$ must be
irreducible.

It remains to consider $\overline{\psi}$. From~\cite{L2}, the
complex permutation character $\rho^0$ of $G$ acting on singular
points has $3$ constituents of degrees $1$, $(2^n-1)(2^{n-1}+2)/3$,
and $(2^{2n}-4)/3$. It follows by Lemma~\ref{main2} that $\psi$ must
be a common constituent of $\rho$ and $\rho^0$. The lemma follows by
using the decomposition of $\overline{\psi}$ from~$\S6$
of~\cite{ST}.
\end{proof}

\textbf{Proof of Theorem \ref{theorem} when $G=O^+_{2n}(2)$}. We
consider four cases as described in Table~\ref{tableO+}:

\medskip

(i) $\ell\neq2,3;\ell\nmid (2^n-1)$: Then we have
$[P]=2^{2n-1}-2^{n-1}\neq0$ and therefore $S(\FF P)\cap T(\FF
P)=\{0\}$. So $\FF P=S(\FF P)\oplus T(\FF P)$. From the proof of
Proposition~\ref{dimensionO+}, we see that $U'_{2^{n-2}}\oplus
U'_{-2^{n-1}}=S(\FF P)$. So
$$\FF P=T(\FF P)\oplus U'_{2^{n-2}}\oplus U'_{-2^{n-1}}.$$
By Proposition~\ref{mainO+}, these direct summands are simple and
their dimensions are given in Proposition \ref{dimensionO+}. In
Table~\ref{tableO+}, $X:=U'_{2^{n-2}}$ and $Y=U'_{-2^{n-1}}$.

\medskip

(ii) $\ell\neq 2,3;\ell\mid (2^n-1)$: Then $T(\FF P)\subset S(\FF
P)$. Since $\ell\neq 3$ and $\ell\mid (2^n-1)$, we have $\ell\nmid
(2^{n-2}-1)$ and therefore $2^{2n-2}-2^{n-1}\neq 2^{n-1}$. Recall
that $T(v)=av=(2^{2n-2}-2^{n-1})v$ for any $v\in T(\FF P)$ and
$T(v)=2^{n-1}v$ for any $v\in U'_{2^{n-2}}$. Therefore
\begin{equation}\label{proof of theoremO+1}
T(\FF P)\cap U'_{2^{n-2}}=\{0\}.
\end{equation}
Hence, by Proposition~\ref{mainO+}, $U'_{2^{n-2}}$ is simple.

If $c$ is a root of the quadratic equation $x^2+(r-s)x+(s-a)=0$,
then
$\sum_{\delta\in\Delta(\alpha)}v_{c,\delta}=(a-s)\alpha+(r-s+c)[\Delta(\alpha)]+s[P]\in
U_c$. Therefore,
$$s[P]=(a-s)\alpha+(r-s+c)[\Delta(\alpha)]+s[P]-(r-s+c)(c\alpha+[\Delta(\alpha)])\in U_c.$$
Recall that $s=2^{2n-3}-2^{n-1}=2^{n-1}(2^{n-2}-1)\neq 0$ as
$\ell\nmid (2^{n-2}-1)$. It follows that $[P]\in U_{2^{n-2}}$ and
$[P]\in U_{-2^{n-1}}$. Equivalently,
\begin{equation}\label{proof of theoremO+2}
T(\FF P)\subseteq U_{2^{n-2}} \text{ and } T(\FF P)\subseteq
U_{-2^{n-1}}.
\end{equation}
Now (\ref{proof of theoremO+1}) and (\ref{proof of theoremO+2})
imply that $U_{2^{n-2}}=U'_{2^{n-2}}\oplus T(\FF P)$. Notice that
$v_{2^{n-2},\alpha}-v_{-2^{n-1},\alpha}=3\cdot 2^{n-2}\alpha$ for
any $\alpha\in P$. So $\FF P=U_{2^{n-2}}+U_{-2^{n-1}}$. It follows
that $\FF P=U'_{2^{n-2}}+U_{-2^{n-1}}$ since $T(\FF P)\subseteq
U_{-2^{n-1}}$. We also have $\dim U'_{2^{n-2}}+\dim U_{-2^{n-1}}\leq
\dim U'_{2^{n-2}}+\dim U'_{-2^{n-1}}+1=\dim S(\FF P)+1=|P|$. So
$$\FF P=U'_{2^{n-2}}\oplus U_{-2^{n-1}}.$$

If $T(\FF P)\nsubseteq U'_{-2^{n-1}}$ then $U_{-2^{n-1}}=T(\FF
P)\oplus U'_{-2^{n-1}}$. It follows that $(T(\FF P)\oplus
U'_{-2^{n-1}})\cap U'_{2^{n-2}}=0$, which leads to a contradiction
since $T(\FF P)\subset S(\FF P)=U'_{-2^{n-1}}\oplus U'_{2^{n-2}}$.
So we have $T(\FF P)\subset U'_{-2^{n-1}}$. We deduce that, by
Proposition~\ref{mainO+}, $U_{-2^{n-1}}$ is uniserial with
composition series
$$0\subset T(\FF P) \subset U'_{-2^{n-1}}\subset
U_{-2^{n-1}}.$$ It is easy to see that
$v_{-2^{n-1},\alpha}-v_{-2^{n-1},\alpha}g=v_{-2^{n-1},\alpha}-v_{-2^{n-1},\alpha
g}\in U'_{-2^{n-1}}$. Therefore $U_{-2^{n-1}}/U'_{-2^{n-1}}$ is
isomorphic to the one-dimensional trivial $\FF G$-module. If we put
$Y:=U'_{-2^{n-1}}/T(\FF P)$, then the socle series of $U_{-2^{n-1}}$
is $\FF-Y-\FF$ (here and after, we use row notation for socle
series). We note that, in Table~\ref{tableO+}, $X:=U'_{2^{n-2}}$.

\medskip

(iii) $\ell=3, n$ even: Then $2^{n-2}=-2^{n-1}=1$. Let $S\subset P$
be the set of all nonsingular points of the form $\langle
e_1+f_1+v\rangle$ with $v\in \langle e_2,e_3,...,e_n\rangle$. It is
easy to check that
$$\sum_{\alpha\in S}v_{1,\alpha}=\sum_{\alpha\in S}(\alpha+[\Delta(\alpha)])=
\sum_{\alpha\in S}\alpha+2^{n-2}\sum_{\alpha\in P\setminus
S}\alpha=[P].$$ Therefore, $[P]\in U_1$. Moreover, since
$|S|=2^{n-1}\neq0$, $[P]\notin U'_1$. Therefore we have $U_1=T(\FF
P)\oplus U'_1$ and $U'_1$ is simple by Proposition~\ref{mainO+}.
Moreover, $U_1$ is the socle of $\FF P$.

Recall that the submodule $U_1$ consists of $\FF$-linear
combinations of $v_{1,\alpha}, \alpha\in P$. Define a bilinear form
$[\cdot,\cdot]$ on $U_1$ by $[v_{1,\alpha},v_{1,\beta}]=\langle
v_{1,\alpha},\beta\rangle$. It is clear that this form is symmetric,
non-singular, and $G$-invariant. Hence, $U_1$ is self-dual. By
Lemma~\ref{liebeck}, $\langle
v_{1,\alpha},v_{1,\beta}\rangle=s=2^{2n-3}-2^{n-1}=0$ for any
$\alpha,\beta\in P$. Hence $\langle U_1,U_1\rangle=0$ or
equivalently, $U_1\subseteq U_1^\perp$. Therefore, $\FF
P/U_1^\perp\cong \Hom_\FF(U_1,\FF)\cong U_1$. We have shown that
$0\subset U_1\subseteq U_1^\perp\subset\FF P$ is a series of $\FF P$
with $\FF P/U_1^\perp\cong U_1\cong T(\FF P)\oplus U'_1$. This and
Lemma \ref{lemmaO+1} imply that $\dim U'_1=(2^n-1)(2^{n-1}-1)/3$ and
$U_1^\perp/U_1$ is simple of dimension $(2^n-1)(2^{n-1}+2)/3-2$.

Set $X:=U'_1$ and $Z:=U_1^\perp/U_1$. Then the composition factors
of $\FF P$ are: $\FF$ (twice), $X$ (twice), and $Z$. By
Proposition~\ref{mainO+}, $S(\FF P)/T(\FF P)$ is uniserial with
socle series $X-Z-X$. Note that $U_1^{'\perp}/U'_1$ has composition
factors: $\FF$ (twice) and $Z$. We will show that
$U_1^{'\perp}/U'_1$ is also uniserial with socle series $\FF-Z-\FF$.
Assume the contrary. Then $Z$ is a direct summand of self-dual
$U_1^{'\perp}/U'_1$. By~(\ref{Q}), $\FF P/\Ker(Q)\cong \Im(Q)$,
which is neither $0$ nor $T(\FF P^0)$. Inspecting the structure of
$\FF P^0$ given in Figure~4 of~\cite{ST}, we see that $\Im(Q)$ has a
uniserial submodule $\FF-Z$. It follows that $\FF P/\Ker(Q)$ also
has a uniserial submodule $\FF-Z$, which leads to a contradiction
since $Z$ is a direct summand of $U_1^{'\perp}/U'_1$. We conclude
that the socle series of $\FF P$ is $(\FF\oplus X)-Z-(\FF\oplus X)$.

\medskip

(iv) $\ell=3, n$ odd: Then $2^{n-2}=-2^{n-1}=2$. Since
$|P|=2^{2n-1}-2^{n-1}\neq0$, $[P]\notin S(\FF P)$ and therefore $\FF
P=T(\FF P)\oplus S(\FF P)$. As in~(iii), we have $U_2=T(\FF P)\oplus
U'_2$ and $U'_2$ is simple. Moreover, $U'_2$ is the socle of $S(\FF
P)$. We also have that $U_2$ as well as $U'_2$ are self-dual. By
Lemma~\ref{liebeck}, $\langle U'_2,U_2\rangle=0$, whence
$U'_2\subseteq U_2^\perp$. Since $[P]\in U_2$, $U_2^\perp \subset
S(\FF P)$ and we obtain $S(\FF P)/U_2^\perp$ $\cong$ $\FF
P/U_2^{'\perp}$ $\cong$ $ \Hom_\FF(U'_2,\FF)\cong U'_2$. Combining
this with Lemma~\ref{lemmaO+1}, we have $\dim
U'_2=(2^n-1)(2^{n-1}-1)/3$ and $U_2^\perp/U'_2$ is simple of
dimension $(2^n-1)(2^{n-1}+2)/3-1$. By Lemma~\ref{main3} and the
self-duality of $S(\FF P)$, $S(\FF P)/U_2^\perp$ is the top layer of
$S(\FF P)$ and therefore $U_2^\perp/U'_2$ is the second layer.
Setting $X=U'_2$ and $Z=U_2^\perp/U'_2$, then the socle series of
$S(\FF P)$ is $X-Z-X$.\hfill$\Box$


\section{The orthogonal groups $O^-_{2n}(2)$}\label{sectionO-}
Let $Q(\cdot)$ be a quadratic form of type $-$ on a vector space $V$
of dimension $2n\geq6$ and $(.,.)$ be a symmetric bilinear form
associated with $Q$ so that $Q(au+bv)=a^2Q(u)+b^2Q(v)+ab(u,v)$ for
any $a,b\in \FF_2,u,v\in V$. Then $G=O^-_{2n}(2)$ is the orthogonal
group of linear transformations of $V$ preserving $Q$. We choose a
basic of $V$ consisting of vectors $\{e_1,...,e_n,f_1,...,f_n\}$ so
that $(e_i,f_j)=\delta_{ij}$, $(e_i,e_j)=(f_i,f_j)=0$ for every
$i,j=1,...,n$,
$Q(e_1)=\cdot\cdot\cdot=Q(e_{n-1})=Q(f_1)=\cdot\cdot\cdot=Q(f_{n-1})=0$,
and $Q(e_n)=Q(f_n)=1$. If $P$ is the set of all nonsingular points
in $V$ then $P=\{\langle\sum_1^n(a_ie_i+b_if_i)\rangle\mid
a_n+b_n+\sum_1^na_ib_i=1\}$ and $|P|=2^{2n-1}+2^{n-1}$.

A similar calculation as in $\S$\ref{sectionO+} shows that
$a=2^{2n-2}+2^{n-1}, b=2^{2n-2}-1, r=2^{2n-3}+2^{n-2},
s=2^{2n-3}+2^{n-1}$. Equation~(\ref{quadratic}) now becomes
$x^2-2^{n-2}x-2^{2n-3}=0$ and it has two roots $-2^{n-2}$ and
$2^{n-1}$.

\begin{proposition}\label{mainO-} Suppose that the characteristic of $\mathbb{F}$ is
odd. Then every nonzero $\FF G$-submodule of $\FF P$ either is
$T(\FF P)$ or contains a graph submodule, which is $U'_{-2^{n-2}}$
or $U'_{2^{n-1}}$.
\end{proposition}

\begin{proof} Define $\Delta_1,\Delta_2,\Delta,\Phi, P_1, V_1$, and $B_w$ similarly to those of $\S 3$.
Consider a subgroup $H<G$ consisting of orthogonal transformations
sending elements of the basis $\{e_1$, $f_1$, $e_2$,
$f_2,...,e_n,f_n\}$ to those of the basis
$\{e_1,f_1+\sum_{i=1}^na_ie_i+\sum_{i=2}^nb_if_i,e_2-b_2e_1,f_2-a_2e_1,...,e_n-b_ne_1,f_n-a_ne_1\}$
respectively, where $a_i,b_i\in\FF_2$ and
$a_1=a_n+b_n+\sum_{i=2}^na_ib_i$. Let $K$ be the subgroup of $H$
consisting of transformations fixing $e_2+f_2$. Then
Lemmas~\ref{delta} and~\ref{phi} remains true as stated. Now we just
argue as in the proof of Proposition~\ref{mainO+}.
\end{proof}

\begin{proposition}\label{dimensionO-} If $\ell=\Char(\FF)\neq 2,3$, then
$$\dim U'_{-2^{n-2}}=\frac{(2^n+1)(2^{n-1}+1)}{3} \text{ and } \dim U'_{2^{n-1}}=\frac{2^{2n}-4}{3}.$$
\end{proposition}

\begin{proof} Similarly to Proposition~\ref{dimensionO+}, we find
$$\dim U'_{-2^{n-2}}+\dim U'_{2^{n-1}}=|P|-1=2^{2n-1}+2^{n-1}-1$$
and
$$
2^{n-2}\dim U'_{2^{n-1}}-2^{n-1}\dim
U'_{-2^{n-2}}=-(2^{2n-2}+2^{n-1}),$$ and the proposition follows.
\end{proof}

In characteristic $0$ we have $\FF P=T(\FF P)\oplus
U'_{-2^{n-2}}\oplus U'_{2^{n-1}}.$ Let $\varphi$ and $\psi$ be
irreducible complex characters of $G$ afforded by $U'_{-2^{n-2}}$
and $U'_{2^{n-1}}$, respectively. Then we have
$\rho=1+\varphi+\psi,$ where $\varphi(1)=(2^n+1)(2^{n-1}+1)/3$ and
$\psi(1)=(2^{2n}-4)/3$.

\begin{lemma}\label{lemmaO-3} Suppose $\ell=3$. Then $\overline{\psi}$ has exactly $2$ constituents of
degrees $(2^n+1)(2^{n-1}-2)/3$ and $(2^n+1)(2^{n-1}+1)/3-1$ if $n$
is even and $3$ constituents of degree $1$, $(2^n+1)(2^{n-1}-2)/3-1$
and $(2^n+1)(2^{n-1}+1)/3-1$ if $n$ is odd.
\end{lemma}

\begin{proof} From~\cite{L2}, we know that the complex permutation
character $\rho^0$ of $G$ acting on singular points has $3$
irreducible constituents of degrees $1$, $(2^n+1)(2^{n-1}-2)/3$, and
$(2^{2n}-4)/3$. Therefore, by Lemma~\ref{main2}, $\psi$ must be a
common constituent of $\rho$ and $\rho^0$. Now the lemma follows
from Corollary~8.10 of~\cite{ST}.
\end{proof}

Let $G_1$ be the stabilizer of $e_n$ and $f_n$ in $G$ so that
$G_1=O^+_{2n-2}(2)\leq G$. Note that $G_1=\Omega^+_{2n-2}(2)\cdot2$
and $G=\Omega^-_{2n}(2)\cdot 2$. Denote by $\omega$ and $\omega_1$
the non-trivial $3$-Brauer linear characters of $G$ and $G_1$,
respectively. Then we have $\omega|_{G_1}=\omega_1$. The two
following lemmas describe the decompositions of restrictions of
$\overline{\rho}$, $\overline{\varphi}$, and $\overline{\psi}$ to
$G_1$ when $\ell=3$.

\begin{lemma}\label{lemmaO-1} Suppose $\ell=3$ and $n$ is even. Then
\begin{enumerate}
\item[(i)] $\overline{\rho}|_{G_1}$ has exactly $7$ trivial constituents,
$3$ constituents $\omega_1$, $5$ constituents of degree
$(2^{n-1}-1)(2^{n-2}-1)/3$, and $7$ constituents of degree
$(2^{n-1}-1)(2^{n-2}+2)/3-1$.
\item[(ii)] $\overline{\varphi}|_{G_1}$ has exactly $5$ constituents of degree $1$, $1$ constituent of degree
$(2^{n-1}-1)(2^{n-2}-1)/3$, and $3$ constituents of degree
$(2^{n-1}-1)(2^{n-2}+2)/3-1$.
\item[(iii)] $\overline{\psi}|_{G_1}$ has exactly $4$
constituents of degree $1$, $4$ constituents of degree
$(2^{n-1}-1)(2^{n-2}-1)/3$, and $4$ constituents of degree
$(2^{n-1}-1)(2^{n-2}+2)/3-1$.
\end{enumerate}
\end{lemma}

\begin{proof} Let $P_1$ and $P_1^0$ be the sets of nonsingular and
singular points, respectively, in $V_1=\langle e_1$, $...$,
$e_{n-1}$, $f_1,...,f_{n-1}\rangle$. We temporarily abuse the
notation by identifying a nonzero vector $v$ with the point
containing it. Then we have
\begin{equation}\label{P}P=\{e_n,f_n,e_n+f_n\}\cup P_1\cup \{ e_n+P_1^0\}\cup\{
f_n+P_1^0\}\cup \{ e_n+f_n+P_1^0\},\end{equation} where the union is
disjoint and by $\{x+Y\}$ we mean $\{x+y\mid y\in Y\}$. Since $G_1$
fixes $e_n,f_n,e_n+f_n$ and the action of $G_1$ on each set
$\{e_n+P_1^0\}$, $\{f_n+P_1^0\}$, or $\{e_n+f_n+P_1^0\}$ is the same
as that of $G_1$ on $P_1^0$, we have the following isomorphism as
$\FF G_1$-modules:
$$\FF P\cong 3\FF\oplus\FF P_1\oplus 3\FF P_1^0.$$
The structure of $\FF P_1$ is given in Table~\ref{tableO+} and the
structure of $\FF P_1^0$ is given in Figure~4 of~\cite{ST}. Hence
part~(i) follows.

We also have $P^0=P^0_1\cup \{ e_n+P_1\}\cup\{ f_n+P_1\}\cup \{
e_n+f_n+P_1\}$. This and (\ref{P}) give us the $\CC
G_1$-isomorphisms:
$$\CC P\cong 3\CC\oplus\CC P_1\oplus 3\CC P_1^0 \text{ and } \CC P^0\cong \CC P^0_1\oplus 3\CC P_1.$$
As $\psi$ is a common constituent of $\rho$ and $\rho^0$, we assume
that $\rho^0=1_G+\varphi^0+\psi$. Also, the characters afforded by
$\CC P_1$ and $\CC P^0_1$ are $1_{G_1}+\varphi_1+\psi_1$ and
$1_{G_1}+\varphi^0_1+\psi_1$, respectively. We then have
$$\left\{\begin {array}{ll}
1_{G_1}+\varphi|_{G_1}+\psi|_{G_1}=3\cdot 1_{G_1}+(1_{G_1}+\varphi_1+\psi_1)+3(1_{G_1}+\varphi^0_1+\psi_1),\\
1_{G_1}+\varphi^0|_{G_1}+\psi|_{G_1}=(1_{G_1}+\varphi^0_1+\psi_1)+3(1_{G_1}+\varphi_1+\psi_1).\\
\end {array} \right.$$
These equations show that the multiplicities of $\varphi_1$ and
$\varphi_1^0$ in $\psi|_{G_1}$ are at most $1$. By comparing the
degrees, the second equation shows that the multiplicity of $\psi_1$
in $\psi|_{G_1}$ is $3$. Hence,
$$\psi|_{G_1}=3\psi_1+\varphi_1+\varphi^0_1+3\cdot 1_{G_1}.$$
The reductions modulo $3$ of $\psi_1$, $\varphi_1$, and
$\varphi_1^0$ are known from $\S\ref{sectionO+}$. Hence (iii) is
proved and so is (ii).
\end{proof}

\begin{lemma}\label{lemmaO-2} Suppose $\ell=3$ and $n\geq 5$ is odd . Then
\begin{enumerate}
\item[(i)] $\overline{\rho}|_{G_1}$ has exactly $14$ trivial constituents,
$3$ constituents $\omega_1$, $5$ constituents of degree
$(2^{n-1}-1)(2^{n-2}-1)/3$, and $7$ constituents of degree
$(2^{n-1}-1)(2^{n-2}+2)/3-2$.
\item[(ii)] $\overline{\varphi}|_{G_1}$ has exactly $8$
constituents of degree $1$, $1$ constituent of degree
$(2^{n-1}-1)(2^{n-2}-1)/3$, and $3$ constituents of degree
$(2^{n-1}-1)(2^{n-2}+2)/3-2$.
\item[(iii)] $\overline{\psi}|_{G_1}$ has exactly $8$
constituents of degree $1$, $4$ constituents of degree
$(2^{n-1}-1)(2^{n-2}-1)/3$, and $4$ constituents of degree
$(2^{n-1}-1)(2^{n-2}+2)/3-2$.
\end{enumerate}
\end{lemma}

\begin{proof} This is similar to the proof of Lemma~\ref{lemmaO-1}.
\end{proof}

\textbf{Proof of Theorem \ref{theorem} when $G=O^-_{2n}(2)$}. As
described in Table~\ref{tableO-}, we consider the following cases.

\medskip

(i) $\ell\neq2,3; \ell\nmid (2^n+1)$: Then we have $S(\FF P)\cap
T(\FF P)=\{0\}$ and
$$\FF P=T(\FF P)\oplus U'_{-2^{n-2}}\oplus U'_{2^{n-1}}.$$
Propositions~\ref{mainO-} and~\ref{dimensionO-} now imply the
theorem in this case. In Table \ref{tableO-}, $X:=U'_{-2^{n-2}}$ and
$Y:=U'_{2^{n-1}}$.

\medskip

(ii) $\ell\neq2,3; \ell\mid (2^n+1)$: Then we have $T(\FF P)\subset
S(\FF P)$. Since $\ell\neq 3$ and $\ell\mid (2^n+1)$, $\ell\nmid
(2^{n-2}+1)$ and therefore $2^{2n-2}+2^{n-1}\neq -2^{n-1}$. Note
that $T(v)=av=(2^{2n-2}+2^{n-1})v$ for any $v\in T(\FF P)$ and
$T(v)=-2^{n-1}v$ for any $v\in U'_{-2^{n-2}}$. It follows that
\begin{equation}\label{proof of theoremO-1}
T(\FF P)\cap U'_{-2^{n-2}}=\{0\}
\end{equation}
Therefore, by Proposition~\ref{mainO-}, $U'_{-2^{n-2}}$ is simple.
Arguing as in the corresponding case for $G=O^+_{2n}(2)$, we have
\begin{equation}\label{proof of theoremO-2}
T(\FF P)\subseteq U_{-2^{n-2}} \text{ and } T(\FF P)\subseteq
U_{2^{n-1}}.
\end{equation}
Now (\ref{proof of theoremO-1}) and (\ref{proof of theoremO-2})
imply that $U_{-2^{n-2}}=U'_{-2^{n-2}}\oplus T(\FF P)$. Recall that
$v_{2^{n-1},\alpha}-v_{-2^{n-2},\alpha}=3\cdot 2^{n-2}\alpha$ for
any $\alpha\in P$. So $\FF P=U_{-2^{n-2}}+U_{2^{n-1}}$. It follows
that $\FF P=U'_{-2^{n-2}}+U_{2^{n-1}}$ since $T(\FF P)\subseteq
U_{2^{n-1}}$. We also have $\dim U'_{-2^{n-2}}+\dim U_{2^{n-1}}\leq
\dim U'_{-2^{n-2}}+\dim U'_{2^{n-1}}+1=\dim S(\FF P)+1=|P|$. So
$$\FF P=U'_{-2^{n-2}}\oplus U_{2^{n-1}}.$$

By Proposition~\ref{mainO-}, $U_{2^{n-1}}$ is uniserial with
composition series $0\subset T(\FF P) \subset U'_{2^{n-1}}\subset
U_{2^{n-1}}$. Since $v_{2^{n-1},\alpha}-v_{2^{n-1},\alpha}g\in
U'_{2^{n-1}}$, $U_{2^{n-1}}/U'_{2^{n-1}}$ is isomorphic to the
one-dimensional trivial $\FF G$-module. If we put
$Y:=U'_{2^{n-1}}/T(\FF P)$, then $\FF-Y-\FF$ is the socle series of
$U_{2^{n-1}}$. In Table~\ref{tableO-}, $X:=U'_{-2^{n-2}}$.

\medskip

(iii) $\ell=3; n$ even: Then we have $-2^{n-2}=2^{n-1}=2$. Since
$|P|=2^{2n-1}+2^{n-1}\neq 0$, $[P]\notin S(\FF P)$ and therefore
$\FF P=T(\FF P)\oplus S(\FF P)$. As proved before,
$s[P]=\sum_{\delta\in
\Delta(\alpha)}v_{2,\delta}-(r-s+2)v_{2,\alpha}\in U_2$ for any
$\alpha\in P$. Since
$|\Delta(\alpha)|-(r-s+2)=(2^{2n-2}+2^{n-1})-(2^{2n-3}+2^{n-2}-2^{2n-3}-2^{n-1}+2)=2\neq
0$, $[P]\notin U'_2$. Applying Proposition \ref{mainO-}, we see that
$U'_2$ is simple and $U'_2$ is the socle of $S(\FF P)$. By
Lemma~\ref{main3}, $U'_2$ is also (isomorphic to) the top layer of
the socle series of $S(\FF P)$.

By Lemma~\ref{liebeck}, $\langle U'_2,U_2\rangle=0$, so
$U'_2\subseteq U_2^\perp$. Since $s=2^{2n-3}+2^{n-1}\neq0$, $[P]\in
U_2$ and therefore $U_2^\perp\subset S(\FF P)$. Moreover, $S(\FF
P)/U_2^\perp\cong \FF P/U_2^{'\perp}\cong \Hom_\FF(U'_2,\FF)\cong
U'_2$ by the self-duality of $U_2$ and $U_2=T(\FF P)\oplus U'_2$. We
have shown that $U'_2$ occurs as a composition factor in $\FF P$
with multiplicity at least $2$. Let $\sigma\in\IBR_3(G)$ be the
irreducible $3$-Brauer character of $G$ afforded by $U'_2$. Then
$\sigma$ is an irreducible constituent of
$\overline{\rho}=1+\overline{\varphi}+\overline{\psi}$ with
multiplicity at leat $2$.

Assume that $2\sigma$ is contained in $\overline{\varphi}$. Then
$\sigma(1)\leq \varphi(1)/2=(2^n+1)(2^{n-1}+1)/6$, which violates
the result on the lower bound of degrees of nontrivial irreducible
characters of $\Omega^-_{2n}(2)$ given in~\cite{Ho}. So $\sigma$ is
a constituent of $\overline{\psi}$. By Lemma~\ref{lemmaO-3},
$\sigma(1)$ is either $(2^n+1)(2^{n-1}-2)/3$ or
$(2^n+1)(2^{n-1}+1)/3-1$ and also $\sigma$ is a constituent of
$\overline{\varphi}$. If $\sigma(1)=(2^n+1)(2^{n-1}-2)/3$, then
again from the result on the lower bound of degrees of nontrivial
irreducible characters of $\Omega^-_{2n}(2)$, all other constituents
of $\overline{\varphi}$ are linear. In particular,
$\overline{\varphi}$ contains
$(2^n+1)(2^{n-1}+1)/3-(2^n+1)(2^{n-1}-2)/3=2^n+1$ linear
constituents (counting multiplicities). This contradicts part (i) of
Lemma~\ref{lemmaO-1} since $2^n+1\geq 17$. So we have
$\sigma(1)=(2^n+1)(2^{n-1}+1)/3-1$. Since $\sigma$ is a constituent
of $\overline{\varphi}$ and $\varphi(1)=(2^n+1)(2^{n-1}+1)/3$,
$\overline{\varphi}$ has only one another constituent which is
linear.

Now we aim to show that this linear constituent is $\omega$. Assume
the contrary that it is trivial. Using Lemma~\ref{lemmaO-1}(ii), we
deduce that ${\sigma}|_{G_1}$ has exactly four linear constituents
(counting multiplicities). Since $\overline{\psi}|_{G_1}$ has
exactly 4 linear constituents and $\sigma$ is contained in
$\overline{\psi}$, the multiplicity of $\omega_1$ in
$\overline{\rho}|_{G_1}$ is twice as that in $\sigma$, contradicting
Lemma~\ref{lemmaO-1}(i).

We have shown that $U_2^\perp/U'_2$ has two composition factors
affording $\omega$ and a character of degree $(2^n+1)(2^{n-1}-2)/3$.
The self duality of $U_2^{'\perp}/U'_2\cong \FF\oplus
U_2^{\perp}/U'_2$ implies that $U_2^\perp/U'_2\cong \omega\oplus Z$,
where $Z$ is a module affording the character of degree
$(2^n+1)(2^{n-1}-2)/3$. Here we denote by the same $\omega$ the
module affording character $\omega$. We note that, in
Table~\ref{tableO-}, $X:=U'_2$ and the socle series of $S(\FF P)$ is
$X-(\omega\oplus Z)-X$.

\medskip

(iv) $\ell=3, n$ odd: The case $n=3$ can be checked directly. So we
assume that $n\geq 5$. We have $-2^{n-2}=2^{n-1}=1$. Also,
$s=2^{2n-3}+2^{n-1}=0$ and hence $\langle U_1,U_1\rangle=0$ by
Lemma~\ref{liebeck}. It follows that $U_1\subseteq U_1^\perp$. Since
$U_1$ is self-dual, $\FF P/U_1^\perp\cong\Hom_\FF(U_1,\FF)\cong
U_1$. Therefore, the nontrivial factor of $U'_1$ occurs with
multiplicities at least $2$ in $\FF P$. Note that, by
Proposition~\ref{mainO-}, this nontrivial factor is either $U'_1$ or
$U'_1/T(\FF P)$. Arguing similarly as in (iii) and using
Lemmas~\ref{lemmaO-2} and~\ref{lemmaO-3}, we again obtain that
$\overline{\varphi}$ has exactly two irreducible constituents of
degrees $1$ and $(2^n+1)(2^{n-1}+1)/3-1$. Combining this with
Lemma~\ref{lemmaO-3}, we find that $\overline{\rho}$ has exactly $3$
linear constituents, $2$ constituents of degree
$(2^n+1)(2^{n-1}+1)/3-1$, and $1$ constituent of degree
$(2^n+1)(2^{n-1}-2)/3-1$.

As $U'_1=U_1\cap S(\FF P)$, $U_1/U'_1$ is isomorphic to the
one-dimensional trivial module. It follows that, if $T(\FF P)\in
U'_1$, $U_1$ would have $\FF$ as a composition factor with
multiplicity $2$. Therefore $\FF$ appears at least $4$ times as a
composition factor in $\FF P$, contradicting the previous paragraph.
So $T(\FF P)\nsubseteq U'_1$ and hence $U'_1$ is simple.

Now we have $\dim U'_1=(2^n+1)(2^{n-1}+1)/3-1$. Also,
$U_1^\perp/U_1$ has two composition factors of degrees $1$ and
$(2^n+1)(2^{n-1}-2)/3-1$. Applying Lemma~\ref{lemmaO-2} and arguing
as in (iii), the factor of degree $1$ must be $\omega$ and therefore
the self-dual $U_1^\perp/U_1\cong Z\oplus\omega$, where $Z$ is the
factor of dimension $(2^n+1)(2^{n-1}-2)/3-1$. Setting $X:=U'_1$.
Then $\FF P$ has composition factors: $\FF$ (twice), $\omega$, $X$
(twice), and $Z$. Then, by Proposition~\ref{mainO-} and
Lemma~\ref{main3}, $S(\FF P)/T(\FF P)$ has socle series
$X-(Z\oplus\omega)-X$.

Now we study the structure of the self-dual module
$U_1^{'\perp}/U'_1$. It has composition factors: $\FF$ (twice),
$\omega$, and $Z$. Inspecting the structure of $\FF P^0$ given in
Figure~7 of~\cite{ST}, we see that any nontrivial quotient of $\FF
P^0$ has the uniserial module $Z-\FF$ as a quotient. Using
Lemma~\ref{main1}, we deduce that $Z$ is not a submodule of
$U_1^{'\perp}/U'_1$. If $\omega$ is not a submodule of
$U_1^{'\perp}/U'_1$ neither, the socle series of the self-dual
module $U_1^{'\perp}/U'_1$ would be $\FF-(Z\oplus \omega)-\FF$ and
hence $\FF P$ does not have any nontrivial submodule isomorphic to a
quotient of $\FF P^0$, violating Lemma~\ref{main1}. So $\omega$ must
be a submodule of $U_1^{'\perp}/U'_1$ and therefore
$U_1^{'\perp}/U'_1$ is direct sum of $\omega$ with a uniserial
module $\FF-Z-\FF$. The structure of $\FF P$ now is determined
completely as described. \hfill$\Box$


\section{The unitary groups in even dimensions $U_{2n}(2)$}

Let $V$ be a vector space of dimension $2n\geq4$ over the field of
$4$ elements $\mathbb{F}_4=\{0,1,\tau,\tau^2\}$, where $\tau$ is a
primitive cubic root of unity. Let $(.,.)$ be a nonsingular
conjugate-symmetric sesquilinear form on $V$, i.e, $(.,.)$ is linear
in the first coordinate and $(u,v)=\overline{(v,u)}$ for any $u,v\in
V$, where $\overline{x}=x^2$ for $x\in \FF_4$. Then $G=U_{2n}(2)$ is
the unitary group of linear transformations of $V$ preserving
$(.,.)$.

We choose a basic of $V$ consisting of vectors
$\{e_1,...,e_n,f_1,...,f_n\}$ so that $(e_i,e_j)=(f_i,f_j)=0$ and
$(e_i,f_j)=\delta_{ij}$ for all $i,j=1,...,n$. Let $P$ be the set of
all nonsingular points in $V$. Then
$P=\{\langle\sum_1^n(a_ie_i+b_if_i)\rangle\mid
\sum_1^n(a_i\overline{b_i}+\overline{a_i}b_i)=1\}$ and
$|P|=(2^{4n-1}-2^{2n-1})/3$. Unlike the study of permutation modules
for orthogonal groups, in this section, for computational
convenience we assume that $\Delta(\alpha)\subset P\setminus
\{\alpha\}$ consists of elements orthogonal to $\alpha$ and
$\Phi(\alpha)\subset P\setminus \{\alpha\}$ consists of elements not
orthogonal to $\alpha$. Then we have $a=(2^{4n-3}+2^{2n-2})/3,
b=2^{4n-3}-2^{2n-2}-1, r=(2^{4n-5}-2^{2n-3})/3,
s=(2^{4n-5}+2^{2n-2}/3$. Equation~(\ref{quadratic}) now becomes
$x^2-2^{2n-3}x-2^{4n-5}=0$ and it has two roots $-2^{2n-3}$ and
$2^{2n-2}$.

\begin{proposition}\label{mainUeven} Suppose that the characteristic of $\mathbb{F}$ is
odd. Then every nonzero $\FF G$-submodule of $\FF P$ either is
$T(\FF P)$ or contains a graph submodule, which is $U'_{-2^{2n-3}}$
or $U'_{2^{2n-2}}$.
\end{proposition}

\begin{proof} We use some ideas from the proof of a similar result for
the permutation module of $G$ acting on singular points (see
\cite{L1}). Let $\phi_1:=\langle e_2+\tau f_2\rangle$,
$\phi_2:=\langle e_1+e_2+\tau f_2\rangle$, $\phi_3:=\langle \tau^2
e_1+ e_2+\tau f_2\rangle$, and $\phi_4:=\langle \tau e_1+e_2+\tau
f_2\rangle$. Also
$\Delta:=\{\langle\sum_{i=1}^n(a_ie_i+b_if_i)\rangle\in P\mid
b_1=1\}$, $\Delta_1:=\{\langle\sum_{i=1}^n(a_ie_i+b_if_i)\rangle\in
P\mid b_1=1, a_2\tau^2+b_2=0\}$,
$\Delta_2:=\{\langle\sum_{i=1}^n(a_ie_i+b_if_i)\rangle\in P\mid
b_1=1, a_2\tau^2+b_2=1\}$,
$\Delta_3:=\{\langle\sum_{i=1}^n(a_ie_i+b_if_i)\rangle\in P\mid
b_1=1,a_2\tau^2+b_2=\tau\}$,
$\Delta_4:=\{\langle\sum_{i=1}^n(a_ie_i+b_if_i)\rangle\in P\mid
b_1=1, a_2\tau^2+b_2=\tau^2\}$, and
$\Phi:=\{\langle\sum_{i=1}^n(a_ie_i+b_if_i)\rangle\in P\mid
b_1=0\}$. Then we have
$$\Delta=\Delta_1\cup\Delta_2\cup\Delta_3\cup\Delta_4 \text{ and } P=\Delta\cup\Phi,$$
where the unions are disjoint and
\begin{equation}\label{Deltaphi}[\Delta(\phi_i)]-[\Delta(\phi_j)]=[\Delta_i]-[\Delta_j]\end{equation} for
any $i,j=1,2,3,4$.

Consider a subgroup $H<G$ consisting of unitary transformations
sending elements of basis $\{e_1$, $f_1$, $e_2$, $f_2,...,e_n,f_n\}$
to those of basis
$\{e_1,f_1+\sum_{i=1}^na_ie_i+\sum_{i=2}^nb_if_i,e_2-\overline{b_2}e_1,f_2-\overline{a_2}e_1,...,e_n-\overline{b_n}e_1,f_n-\overline{a_n}e_1\}$
respectively, where $a_i,b_i\in\FF_4$ and
$a_1+\overline{a_1}=\sum_{i=2}^n(\overline{a_i}b_i+a_i\overline{b_i})$.
Let $K$ be the subgroup of $H$ consisting of transformations fixing
$\phi_1$. Let $P_1$ be the set of nonsingular points in $V_1=\langle
e_2,f_2,...,e_n,f_n\rangle$. For each $\langle w\rangle\in P_1$, we
define $B_{\langle w\rangle}=\{\langle w\rangle, \langle
e_1+w\rangle,\langle \tau e_1+w\rangle, \langle
\tau^2e_1+w\rangle\}$. Similarly as in Lemmas~\ref{delta} and
\ref{phi}, we have
\begin{enumerate}
\item[(i)] $|H|=2^{4n-3}$, $|K|=2^{4n-5}$, $|\Delta|=2^{4n-3}$, and $|\Delta_1|=|\Delta_2|=|\Delta_3|=|\Delta_4|=2^{4n-5}$;
\item[(ii)] $H$ acts regularly on $\Delta$ and $K$ has $4$ orbits
$\Delta_1,\Delta_2, \Delta_3,\Delta_4$ on $\Delta$;
\item[(iii)] $\Phi=\bigcup_{\langle w\rangle\in P_1}B_{\langle w\rangle}$;
\item[(iv)] $K$ fixes $B_{\phi_1}$ point-wise and is transitive
on $B_w$ for every $\phi_1\neq\langle w\rangle\in P_1$;
\item[(v)] $H$ acts transitively on $B_w$ for every $\langle w\rangle\in P_1$.
\end{enumerate}

Suppose that $U$ is a nonzero submodule of $\FF P$. Assume that
$U\neq T(\FF P)$. As in Proposition~\ref{mainO+}, we can show that
$U$ contains an element of the form
$$u=a\phi_1+b\phi_2+c\phi_3+d\phi_4+\sum_{\delta\in P\backslash\{\phi_1,\phi_2,\phi_3,\phi_4\}}a_\delta\delta,$$
where $a,b,c,d,a_\delta\in\FF$ and $a\neq b$. Take an element $g\in
G$ such that $e_1g=e_1$ and $(e_2+\tau f_2)g=e_1+e_2+\tau f_2$. Then
$\phi_1 g=\phi_2$, $\phi_2 g=\phi_1$, $\phi_3 g=\phi_4$, and $\phi_4
g=\phi_3$ . So we have
$$u-ug=(a-b)(\phi_1-\phi_2)+(c-d)(\phi_3-\phi_4)+\sum_{\delta\in P\backslash\{\phi_1,\phi_2,\phi_3,\phi_4\}}b_\delta\delta\in
U\cap S(\FF P),$$ where $b_\delta\in\FF$. Now arguing as in
Proposition~\ref{mainO+}, we get an element
$$u_5=(\phi_1-\phi_2)+e(\phi_3-\phi_4)+f[\Delta_1]+f'[\Delta_2]+f''[\Delta_3]+f'''[\Delta_4]\in
U\cap S(\FF P),$$ where $e,f,f',f'',f'''\in \FF$. Interchanging
$\phi_3$ and $\phi_4$ if necessary, we suppose that $e\neq -1$.

Let $g$ be an element of $H$ such that $(e_2+\tau f_2)g=\tau^2
e_1+e_2+\tau f_2$. Then $\phi_1 g=\phi_3$, $\phi_3 g=\phi_1$,
$\phi_2 g=\phi_4$, and $\phi_4 g=\phi_2$. Also, $\Delta_1
g=\Delta_3$, $\Delta_3 g=\Delta_1$, $\Delta_2 g=\Delta_4$, and
$\phi_4 g=\Delta_2$. Therefore,
$$u_6:=u_5g=(\phi_3-\phi_4)+e(\phi_1-\phi_2)+f[\Delta_3]+f'[\Delta_4]+f''[\Delta_1]+f'''[\Delta_2]\in
U\cap S(\FF P).$$ Note that $f+f'+f''+f'''=0$. Therefore, if we set
$u_7:=(u_5+u_6)/(1+e)$, then
$$u_7=(\phi_1-\phi_2+\phi_3-\phi_4)+(t+t'')([\Delta_1]-[\Delta_2]+[\Delta_3]-[\Delta_4])\in U\cap S(\FF
P),$$ where $t=f/(1+e)$ and $t''=f''/(1+e)$.

{\bf Case 1:} If $t+t''\neq 0$ then by~(\ref{Deltaphi}), we have
$u_8:=u_7/(t+t'')=v_{c,\phi_1}-v_{c,\phi_2}+v_{c,\phi_3}-v_{c,\phi_4}\in
U$, where $c=1/(t+t'')$. Consider an element $g\in G$ such that
$e_1g=e_1$ and $(e_2+\tau f_2)g=\tau^2(e_1+e_2+\tau f_2)$. It is
easy to check that $\phi_1g=\phi_2, \phi_2g=\phi_3, \phi_3g=\phi_1$,
and $\phi_4g=\phi_4$. We deduce that
$u_9:=u_8g=v_{c,\phi_2}-v_{c,\phi_3}+v_{c,\phi_1}-v_{c,\phi_4}\in
U$. It follows that $u_{10}:=u_8+u_9=2(v_{c,\phi_1}-v_{c,\phi_4})\in
U$. Hence $U'_c\subseteq U$, which implies that $U$ contains either
$U'_{-2^{2n-3}}$ or $U'_{2^{2n-2}}$ by Lemma~\ref{liebeck}.

{\bf Case 2:} If $t+t''=0$ then $u_7=\phi_1-\phi_2+\phi_3-\phi_4\in
U$. Let $g\in G$ be the same element as in Case 1. We have
$u_7g=\phi_2-\phi_3+\phi_1-\phi_4\in U$. It follows that
$u_9:=u_7+u_8=2(\phi_1-\phi_4)\in U$. Therefore $\alpha-\beta\in U$
for every $\alpha,\beta\in P$. In other words, $U\supseteq S(\FF
P)$, which implies that $U$ contains both $U'_{-2^{2n-3}}$ and
$U'_{2^{2n-2}}$, as wanted.
\end{proof}

\begin{proposition}\label{dimensionUeven} If $\ell=\Char(\FF)\neq 2,3$, then
$$\dim U'_{-2^{2n-3}}=\frac{(2^{2n}-1)(2^{2n-1}+1)}{9} \text{ and } \dim U'_{2^{2n-2}}=\frac{(2^{2n}+2)(2^{2n}-4)}{9}.$$
\end{proposition}

\begin{proof} As in Proposition~\ref{dimensionO+}, we have
$$\dim U'_{-2^{2n-3}}+\dim U'_{2^{2n-2}}=|P|-1=\frac{2^{4n-1}-2^{2n-1}}{3}-1$$
and
$$2^{2n-3}\dim U'_{2^{2n-2}}-2^{2n-2}\dim
U'_{-2^{2n-3}}=-a=-\frac{2^{4n-3}+2^{2n-2}}{3},$$ which imply the
proposition.
\end{proof}

In the complex case $\ell=0$, as before, we have
$$\FF P=T(\FF P)\oplus U'_{-2^{2n-3}}\oplus U'_{2^{2n-2}} \text{ and } \rho=1+\varphi+\psi,$$
where $\varphi$ and $\psi$ are irreducible complex characters of $G$
afforded by $U'_{-2^{2n-3}}$ and $U'_{2^{2n-2}}$, respectively. Note
that $\varphi(1)=(2^{2n}-1)(2^{2n-1}+1)/9$,
$\psi(1)=(2^{2n}+2)(2^{2n}-4)/9$.

We now study the decompositions of $\overline{\varphi}$ and
$\overline{\psi}$ into irreducible Brauer characters when $\ell=3$.
We note that, when $\ell=3$, $-2^{2n-3}=2^{2n-2}=1$ and therefore
$\FF P$ has only one graph submodule $U'_1$.

\begin{lemma}\label{lemmaUeven1} Suppose $\ell=3$. Then $\overline{\psi}$ has exactly $2$
constituents of degrees $(2^{2n}-1)(2^{2n-1}-2)/9$ and
$(2^{2n}-1)(2^{2n-1}+1)/9-1$ if $3\nmid n$ and $3$ constituents of
degrees $1$, $(2^{2n}-1)(2^{2n-1}-2)/9$, and
$(2^{2n}-1)(2^{2n-1}+1)/9-2$ if $3\mid n$.
\end{lemma}

\begin{proof} From~\cite{L1}, we know that the complex permutation
character $\rho^0$ of $G$ acting on singular points has $3$
constituents of degrees $1$, $(2^{2n}-1)(2^{2n-1}+4)/3$, and
$(2^{2n}+2)(2^{2n}-4)/9$. Therefore, by Lemma~\ref{main2}, $\psi$
must be a common constituent of $\rho$ and $\rho^0$. Now the lemma
follows from Corollary~4.5 of~\cite{ST}.
\end{proof}

\begin{lemma}\label{lemmaUeven2} Suppose $\ell=3$. Then
\begin{enumerate}
\item[(i)] $\overline{\varphi}$ has exactly $2$ constituents of
degrees $(2^{2n}-1)(2^{2n-1}-2)/9$ and $(2^{2n}-1)/3$.
\item[(ii)] $U'_1$ is simple and $\dim U'_1=(2^{2n}-1)(2^{2n-1}-2)/9$.
\end{enumerate}
\end{lemma}

\begin{proof}
Let $S$ be the set of nonsingular points of the form $\langle
e_1+\tau f_1+v\rangle$ where $v\in \langle e_2,e_3,...,e_n\rangle$.
It is easy to see that
$$\sum_{\alpha\in S}v_{1,\alpha}=\sum_{\alpha\in S}(\alpha+[\Delta(\alpha)])=[P].$$
Therefore, $T(\FF P)\subset U_1$. Moreover, $|S|=4^{n-1}\neq 0$,
whence $T(\FF P)\nsubseteq U'_1$. By Proposition~\ref{mainUeven},
$U_1=U'_1\oplus T(\FF P)$ and $U'_1$ is simple, proving the first
part of (ii). Since $U_1$ as well as $U'_1$ are self-dual, $\FF
P/U_1^{'\perp}\cong \Hom_\FF(U'_1,\FF)\cong U'_1$. Using the fact
$U'_1\subseteq U_1^{'\perp}$ from Lemma~\ref{liebeck}, we have that
$U'_1$ occurs at least twice as a composition factor of $\FF P$.

Denote $E:=\langle e_1,...,e_n\rangle$ and $F:=\langle
f_1,...,f_n\rangle$. Let $Q:=\Stab_G(E)$, the stabilizer of $E$ in
$G$. For any point $\langle f\rangle$ in $F$, the action of $Q$ on
the set of $2^{2n-1}$ nonsingular points of the form $\langle
f+u\rangle$ with $u\in E$ is equivalent to that on the set of
$2^{2n-1}$ singular points of the form $\langle f+u\rangle$ with
$u\in E$. Therefore, if $P(E)$ denotes the set of (singular) points
in $E$, we have an $\FF Q$-isomorphism:
$$\CC P^0\cong \CC P\oplus \CC P(E).$$
Notice that $Q$ acts doubly transitive on $P(E)$ and its character
is $1_{Q}+\tau$, where $\tau$ is irreducible of degree
$(2^{2n}-1)/3-1$. Therefore, the above isomorphism gives:
$$\rho^0|_Q=\rho|_Q+1_Q+\tau.$$
Using Proposition~4.2 of \cite{ST}, we have the decomposition of
$\rho^0|_Q$ into irreducible constituents:
$\rho^0|_Q=2(1+\tau)+\zeta+\sigma_0+\sigma_1$ where
$\zeta(1)=(2^{2n}-1)/3$, $\sigma_0(1)=(2^{2n}-1)(2^{2n-1}-2)/9$, and
$\sigma_1(1)=2(2^{2n}-1)(2^{2n-1}-2)/9$. It follows that
$$\rho|_Q=1_Q+\tau +\zeta +\sigma_0+\sigma_1.$$
By comparing degrees, we have
$$\varphi|_Q=\zeta+\sigma_0 \text{ and } \psi|_Q=\tau+\sigma_1.$$
By Proposition~4.4 of \cite{ST}, we know that $\overline{\zeta}$ and
$\overline{\sigma_0}$ are irreducible when $\ell=3$. It follows that
$\overline{\varphi}$ has at most two constituents. Suppose
$\overline{\varphi}$ is irreducible. Combining this with
Lemma~\ref{lemmaUeven1}, we see that $\overline{\rho}$ has no
constituent with multiplicity $\geq 2$, which contradicts the fact
that $U'_1$ occurs as a composition factor of $\FF P$ at least
twice. Therefore, $\overline{\varphi}$ has exactly two constituents,
and their degrees are degrees $\zeta(1)=(2^{2n}-1)/3$ and
$\sigma_0(1)=(2^{2n}-1)(2^{2n-1}-2)/9$. This shows that $\dim
U'_1=(2^{2n}-1)(2^{2n-1}-2)/9$ and the lemma is proved.
\end{proof}

\textbf{Proof of Theorem \ref{theorem} when $G=U_{2n}(2)$}. As
described in Table \ref{tableUeven}, we consider the following
cases.

\medskip

(i) $\ell\neq2,3; \ell\nmid (2^{2n}-1)$: Then we have
$|P|=(2^{4n-1}-2^{2n-1})/3\neq 0$ and hence $S(\FF P)\cap T(\FF
P)=\{0\}$. We then obtain $\FF P=T(\FF P)\oplus U'_{-2^{2n-3}}\oplus
U'_{2^{2n-2}}.$ Propositions~\ref{mainUeven}
and~\ref{dimensionUeven} now imply the theorem in this case. In
Table \ref{tableUeven}, $X:=U'_{-2^{2n-3}}$ and $Y:=U'_{2^{2n-2}}$.

\medskip

(ii) $\ell\neq2,3; \ell\mid (2^{2n}-1)$: Then $T(\FF P)\subset S(\FF
P)$. As before, we have $T(\FF P)\cap U'_{-2^{2n-3}}=\{0\}$ and
therefore, by Proposition~\ref{mainUeven}, $U'_{-2^{2n-3}}$ is
simple. We also have $\FF P=U'_{-2^{2n-3}}\oplus U_{2^{2n-2}}.$
Moreover, $U_{2^{2n-2}}$ is uniserial with composition series
$0\subset T(\FF P) \subset U'_{2^{2n-2}}\subset U_{2^{2n-2}}$ and
socle series $\FF-Y-\FF$, where $Y:=U'_{2^{2n-2}}/T(\FF P)$. We note
that, in Table~\ref{tableUeven}, $X:=U'_{-2^{2n-3}}$.

\medskip

(iii) $\ell=3; 3\mid n$: Then $s=(2^{4n-5}+2^{2n-2})/3=0$.
Therefore, by Lemma~\ref{liebeck}, $\langle U_1,U_1\rangle=0$,
whence $U_1\subseteq U_1^\perp$. From the proof of
Lemma~\ref{lemmaUeven2}, we know that $U_1=T(\FF P)\oplus U'_1$.
Hence, by Lemmas~\ref{lemmaUeven1} and~\ref{lemmaUeven2},
$U_1^\perp/U_1$ has two composition factors of degrees
$(2^{2n}-1)/3$ and $(2^{2n}-1)(2^{2n-1}+1)/9-2$, which we denote by
$W_1$ and $W_2$, respectively. The self-duality of $U_1^\perp/U_1$
implies that $U_1^\perp/U_1\cong W_1\oplus W_2$. Put $Z:=U'_1$. By
Proposition~\ref{mainUeven}, $S(\FF P)/T(\FF P)$ is uniserial with
socle series $Z-(W_1\oplus W_2)-Z$.

Next, the self-dual module $U_1^{'\perp}/U'_1$ has composition
factors: $\FF$ (twice), $W_1$, and $W_2$. Inspecting the structure
of $\FF P^0$ from Figure~1 of~\cite{ST}, we see that any nontrivial
quotient of $\FF P^0$ has the uniserial module $W_2-\FF$ as a
quotient. It follows by Lemma~\ref{main1} that $W_2$ is not a
submodule of $U_1^{'\perp}/U'_1$. If $W_1$ is not a submodule of
$U_1^{'\perp}/U'_1$ neither, the socle series of the self-dual
module $U_1^{'\perp}/U'_1$ would be $\FF-(W_1\oplus W_2)-\FF$ and
hence $\FF P$ does not have any nontrivial submodule isomorphic to a
quotient of $\FF P^0$, a contradiction by Lemma~\ref{main1}. So
$W_1$ must be a submodule of $U_1^{'\perp}/U'_1$ and therefore
$U_1^{'\perp}/U'_1$ is direct sum of $W_1$ and a uniserial module
$\FF-W_2-\FF$. The structure of $\FF P$ now is determined as
described.

\medskip

(iv) $\ell=3, 3\nmid n$: Then $|P|=(2^{4n-1}-2^{2n-1})/3\neq0$.
Therefore, $\FF P=T(\FF P)\oplus S(\FF P)$. Again by
Lemma~\ref{liebeck}, $\langle U'_1,U_1\rangle=0$. Consider a series
of $S(\FF P)$:
$$0\subset U'_1\subseteq U_1^\perp\subseteq S(\FF P).$$
We have $S(\FF P)/U_1^\perp\cong\FF P/U_1^{'\perp}\cong U'_1$.
Therefore, by Lemmas~\ref{lemmaUeven1} and~\ref{lemmaUeven2},
$U_1^\perp/U'_1$ has two composition factors of degrees
$(2^{2n}-1)/3$ and $(2^{2n}-1)(2^{2n-1}+1)/9-1$, which again we
denote by $W_1$ and $W_2$, respectively. It is easy to see that
$U_1^\perp/U'_1$ is self-dual from the self-duality of
$U_1^{'\perp}/U'_1$. Therefore $U_1^\perp/U'_1\cong W_1\oplus W_2$.
Putting $Z:=U'_1$, the socle series of $S(\FF P)$ is $Z-(W_1\oplus
W_2)-Z$.\hfill$\Box$


\section{The unitary groups in odd dimensions $U_{2n+1}(2)$}

Let $V$ be a vector space of dimension $2n+1\geq5$ over the field of
$4$ elements $\mathbb{F}_4=\{0,1,\tau,\tau^2\}$, where $\tau$ is a
primitive cubic root of unity. Let $(.,.)$ be a nonsingular
conjugate-symmetric sesquilinear form on $V$. Then $G=U_{2n+1}(2)$
is the unitary group of linear transformations of $V$ preserving
$(.,.)$.

We choose a basic of $V$ consisting of vectors
$\{e_1,...,e_n,f_1,...,f_n,g\}$ so that
$(e_i,e_j)=(f_i,f_j)=(e_i,g)=(f_i,g)=0$, $(e_i,f_j)=\delta_{ij}$,
and $(g,g)=1$ for all $i,j=1,...,n$. Let $P$ be the set of
nonsingular points in $V$. Then $P=\{\langle
cg+\sum_1^n(a_ie_i+b_if_i)\rangle\mid
c\overline{c}+\sum_1^n(a_i\overline{b_i}+\overline{a_i}b_i)\neq 0\}$
and $|P|=(2^{4n+1}+2^{2n})/3$. We assume that $\Delta(\alpha)\subset
P\setminus \{\alpha\}$ consists of elements orthogonal to $\alpha$
and $\Phi(\alpha)\subset P\setminus \{\alpha\}$ consists of elements
not orthogonal to $\alpha$. Then we have $a=(2^{4n-1}-2^{2n-1})/3,
b=2^{4n-1}+2^{2n-1}-1, r=(2^{4n-3}+2^{2n-2})/3,
s=(2^{4n-3}-2^{2n-1})/3$. Equation~(\ref{quadratic}) now becomes
$x^2+2^{2n-2}x-2^{4n-3}=0$ and it has two roots $2^{2n-2}$ and
$-2^{2n-1}$.

\begin{proposition}\label{mainUodd} Suppose that the characteristic of $\mathbb{F}$ is
odd. Then every nonzero $\FF G$-submodule of $\FF P$ either is
$T(\FF P)$ or contains a graph submodule, which is $U'_{2^{2n-2}}$
or $U'_{-2^{2n-1}}$.
\end{proposition}

\begin{proof} Define $\phi_1,\phi_2,\phi_3,\phi_4$, $\Delta_1,\Delta_2,\Delta_3,\Delta_4,\Delta,
\Phi, P_1, V_1$, and $B_{\langle w\rangle}$ in a manner similar to
that of the proof of Proposition \ref{mainUeven}. Consider a
subgroup $H<G$ consisting of unitary transformations sending
elements of the basis $\{e_1$, $f_1$, $e_2$, $f_2,...,e_n,f_n,g\}$
to those of the basis
$\{e_1,f_1+\sum_{i=1}^na_ie_i+\sum_{i=2}^nb_if_i+cg,e_2-\overline{b_2}e_1,f_2-
\overline{a_2}e_1,...,e_n-\overline{b_n}e_1,f_n-\overline{a_n}e_1,g-\overline{c}e_1\}$
respectively, where $a_i,b_i,c\in\FF_4$ and
$a_1+\overline{a_1}+c\overline{c}=\sum_{i=2}^n(\overline{a_i}b_i+a_i\overline{b_i})$.
Let $K$ be the subgroup of $H$ consisting of transformations fixing
$\phi_1$. Now we argue as in the proof of Proposition
\ref{mainUeven}.
\end{proof}

\begin{proposition}\label{dimensionUodd} If $\ell=\Char(\FF)\neq 2,3$, then
$$\dim U'_{2^{2n-2}}=\frac{(2^{2n+1}+1)(2^{2n}-1)}{9} \text{ and } \dim U'_{-2^{2n-1}}=\frac{(2^{2n+1}-2)(2^{2n+1}+4)}{9}.$$
\end{proposition}

\begin{proof} As in Proposition~\ref{dimensionO+}, we have
$$\dim U'_{2^{2n-2}}+\dim U'_{-2^{2n-1}}=|P|-1=\frac{2^{4n+1}+2^{2n}}{3}-1$$
and
$$2^{2n-1}\dim U'_{2^{2n-2}}-2^{2n-2}\dim
U'_{-2^{2n-1}}=-a=-\frac{2^{4n-1}-2^{2n-1}}{3},$$ which imply the
proposition.
\end{proof}

In the complex case $\ell=0$, we have
$$\FF P=T(\FF P)\oplus U'_{2^{2n-2}}\oplus U'_{-2^{2n-1}} \text{ and } \rho=1+\varphi+\psi,$$
where $\varphi$ and $\psi$ are irreducible complex characters of $G$
afforded by $U'_{2^{2n-2}}$ and $U'_{-2^{2n-1}}$, respectively. Then
we have $\varphi(1)=(2^{2n+1}+1)(2^{2n}-1)/9$,
$\psi(1)=(2^{2n+1}-2)(2^{2n+1}+4)/9$.

\begin{lemma}\label{lemmaUodd1} Suppose $\ell=3$. Then
\begin{enumerate}
\item[(i)] when $3\nmid n$, $\overline{\psi}$ has exactly $1$ trivial constituent,
$2$ equal constituents of degree $(2^{2n+1}-2)/3$, $1$ constituent
of degree $(2^{2n+1}+1)(2^{2n}-4)/9$, and $1$ constituent of degree
$(2^{2n+1}+1)(2^{2n}-1)/9$;
\item[(ii)] when $3\mid n$, $\overline{\psi}$ has exactly $2$ trivial
constituent, $2$ equal constituents of degree $(2^{2n+1}-2)/3$, $1$
constituent of degree $(2^{2n+1}+1)(2^{2n}-4)/9-1$, and $1$
constituent of degree $(2^{2n+1}+1)(2^{2n}-1)/9$.
\end{enumerate}
\end{lemma}

\begin{proof} From~\cite{L2}, we know that the complex permutation
character $\rho^0$ of $G$ acting on singular points has $3$
constituents of degrees $1$, $(2^{2n+1}+1)(2^{2n}-4)/3$, and
$(2^{2n+1}-2)(2^{2n+1}+4)/9$. Therefore, by Lemma~\ref{main2},
$\psi$ must be a common constituent of $\rho$ and $\rho^0$. Now the
lemma follows from Corollary~5.6 of~\cite{ST}.
\end{proof}

We will show that $\overline{\varphi}$ is irreducible in any case.
In order to do that, we need to recall some results about \emph{Weil
characters} of unitary groups. The Weil representation of $SU_m(q)$
with $m\geq 3$ and $q$ a prime power is of degree $q^m$. Its
afforded character decomposes into a sum of $q+1$ irreducibles ones
where one of them has degree $(q^m+(-1)^mq)/(q+1)$ and the others
have degree $(q^m-(-1)^m)/(q+1)$. These irreducible characters are
called (complex) Weil characters of $SU_m(q)$. Each Weil character
of $SU_m(q)$ extends to $q+1$ distinct Weil characters of $U_m(q)$
(see Lemma~4.7 of~\cite{TZ}). All nonlinear constituents of the
reduction modulo $\ell$ of a complex Weil character are called
($\ell$-Brauer) Weil characters. It is well-known (for instance, see
Theorem~16 of~\cite{HM}) that, when $\ell\nmid q$, any Weil
character lifts to characteristic $0$ and any irreducible Brauer
character of degree either $(q^m+(-1)^mq)/(q+1)$ or
$(q^m-(-1)^m)/(q+1)$ is a Weil character. The following lemma is a
consequence of Lemma~4.2 of~\cite{TZ}.

\begin{lemma}\label{lemmaUodd2} Suppose that $\ell$ is odd. Then the restriction of
a $\ell$-Brauer Weil character of $U_{2n+1}(2)$ of degree
$(2^{2n+1}-2)/3$ to $U_{2n}(2)$ is a sum of two Weil characters of
degree $(2^{2n}-1)/3$.
\end{lemma}

We note that, when $\ell=3$, $2^{2n-2}=-2^{2n-1}=1$ and therefore
$\FF P$ has only one graph submodule $U'_1$.

\begin{lemma}\label{lemmaUodd3} Suppose $\ell=3$. Then
\begin{enumerate}
\item[(i)] $U'_1$ is simple and it occurs as a composition
factor of $\FF P$ with multiplicity $2$.
\item[(ii)] $U'_1$ affords the irreducible character $\overline{\varphi}$.
In particular, $\overline{\varphi}(1)=\dim
U'_1=(2^{2n+1}+1)(2^{2n}-1)/9$.
\end{enumerate}
\end{lemma}

\begin{proof} Let $S$ be the set of nonsingular points in $V$ of the form $\langle
g+v\rangle$ where $v\in \langle e_1,...,e_n\rangle$. It is easy to
see that
$$\sum_{\alpha\in S}v_{1,\alpha}=\sum_{\alpha\in S}(\alpha+[\Delta(\alpha)])=[P].$$
Therefore, $T(\FF P)\subset U_1$. Moreover, $|S|=4^{n}\neq 0$. It
follows that $T(\FF P)\nsubseteq U'_1$. By
Proposition~\ref{mainUodd}, $U'_1$ is simple and $U_1=U'_1\oplus
T(\FF P)$, proving the first part of~(i). By Lemma~\ref{liebeck},
$\langle U'_1,U'_1\rangle=0$ and therefore $U'_1\subseteq
U_1^{'\perp}$. Since $U_1$ as well as $U'_1$ are self-dual, $\FF
P/U_1^{'\perp}\cong \Hom_\FF(U'_1,\FF)\cong U'_1$. Hence $U'_1$
occurs at least twice as a composition factor of $\FF P$.

Let $P_1$ and $P_1^0$ be the sets of nonsingular and singular
points, respectively, in $V_1=\langle e_1$, $...$, $e_{n}$,
$f_1,...,f_{n}\rangle$. Since any nonsingular point in $V$ is either
$\langle g\rangle$, a nonsingular point in $V_1$, or a point of the
form $\langle g+v \rangle$ where $v$ is a nonzero singular vector in
$V_1$, we have the following isomorphism of $\FF U_{2n}(2)$-modules:
$$\FF P\cong \FF\oplus\FF P_1\oplus 3\FF P_1^0.$$ Combining this
isomorphism with the structures of $\FF P_1$ given in
Table~\ref{tableUeven} and $\FF P_1^0$ given in Figure~1
of~\cite{ST}, we see that $\overline{\rho}|_{U_{2n}(2)}$ has exactly
$4$ constituents of degree $(2^{2n}-1)/3$ and at most $15$ trivial
constituents ($\FF P_1$ has $1$ constituent of degree $(2^{2n}-1)/3$
and at most $2$ trivial constituents; $\FF P_1^0$ has $1$
constituent of degree $(2^{2n}-1)/3$ and at most $4$ trivial
constituents). We notice that the constituents of degree
$(2^{2n}-1)/3$ are Weil characters of $U_{2n}(2)$. By
Lemma~\ref{lemmaUodd1}, $\overline{\psi}$ has $2$ constituents of
degree $(2^{2n+1}-2)/3$. Again these constituents are Weil
characters of $G$ and therefore their restrictions to $U_{2n}(2)$
give us $4$ Weil characters of degree $(2^{2n}-1)/3$ by
Lemma~\ref{lemmaUodd2}. We have shown that $\overline{\varphi}$ has
no constituent which is a Weil character.

We will prove~(ii) by contradiction. Assume that
$\overline{\varphi}$ is reducible. By Lemma~\ref{lemmaUodd1} and the
fact that $U'_1$ occurs as a composition factor of $\FF P$ at least
twice, $\overline{\varphi}$ must have a constituent of degree
$(2^{2n+1}+1)(2^{2n}-4)/9$ when $3\nmid n$ or degree
$(2^{2n+1}+1)(2^{2n}-4)/9-1$ when $3\mid n$. Therefore all other
constituents of $\overline{\varphi}$ are of degree $\leq
(2^{2n+1}+1)/3+1$ since $\varphi(1)=(2^{2n+1}+1)(2^{2n}-1)/9$. Using
Theorem~2.7 of~\cite{GMST}, we deduce that these constituents are
either linear or Weil characters. The latter case does not happen
from the previous paragraph. So $\overline{\varphi}$ has at least
$(2^{2n+1}+1)/3$ linear constituents, contradicting to the fact that
$\overline{\rho}|_{U_{2n}(2)}$ has at most $15$ trivial
constituents, provided that $n\geq 3$. The case $n=2$ can be handled
easily by using~\cite{Atl2}.
\end{proof}

\medskip

\textbf{Proof of Theorem \ref{theorem} when $G=U_{2n+1}(2)$}. As
described in Table~\ref{tableUodd}, we consider the following cases.

\medskip

(i) $\ell\neq2,3; \ell\nmid (2^{2n+1}+1)$: Then we have
$|P|=(2^{4n+1}+2^{2n})/3\neq 0$ and therefore $S(\FF P)\cap T(\FF
P)=\{0\}$. As before, we have $\FF P=T(\FF P)\oplus
U'_{2^{2n-2}}\oplus U'_{-2^{2n-1}},$ which implies the theorem in
this case by Propositions~\ref{mainUodd} and~\ref{dimensionUodd}. In
Table~\ref{tableUodd}, $X:=U'_{2^{2n-2}}$ and $Y:=U'_{-2^{2n-1}}$.

\medskip

(ii) $\ell\neq2,3; \ell\mid (2^{2n+1}+1)$: Then $T(\FF P)\subset
S(\FF P)$. As before, we have $T(\FF P)\cap U'_{2^{2n-2}}=\{0\}$ and
therefore, by Proposition~\ref{mainUodd}, $U'_{2^{2n-2}}$ is simple.
We also have $\FF P=U'_{2^{2n-2}}\oplus U_{-2^{2n-1}}.$ Moreover,
$U_{-2^{2n-1}}$ is uniserial with composition series $0\subset T(\FF
P) \subset U'_{-2^{2n-1}}\subset U_{-2^{2n-1}}$ and socle series
$\FF-Y-\FF$, where $Y:=U'_{-2^{2n-1}}/T(\FF P)$. In Table
\ref{tableUodd}, $X:=U'_{2^{2n-2}}$.

\medskip

(iii) $\ell=3; 3\mid n$: In the context of Lemmas~\ref{lemmaUodd1}
and~\ref{lemmaUodd3}, here and after, we will denote by $\FF$, $Z,
X$ the simple $\FF G$-modules affording the $3$-Brauer characters of
degrees $1$, $(2^{2n+1}-2)/3$, $(2^{2n+1}+1)(2^{2n}-1)/9$,
respectively. Also, let $W$ be the irreducible $\FF G$-module
affording the character of degrees $(2^{2n+1}+1)(2^{2n}-4)/9-1$ when
$3 \mid n$ and $(2^{2n+1}+1)(2^{2n}-4)/9$ when $3\nmid n$. By
Lemma~\ref{lemmaUodd3}, we may choose $X=U'_1$.

Also from Lemmas~\ref{lemmaUodd1} and~\ref{lemmaUodd3}, when $3\mid
n$, $\FF P$ has composition factors: $\FF$ ($3$ times), $Z$ (twice),
$X$ (twice), and $W$. Moreover, $|P|=(2^{4n+1}+2^{2n})/3\neq 0$ and
therefore $\FF P=T(\FF P)\oplus S(\FF P)$. It follows that, by
Proposition~\ref{mainUodd}, $X=U'_1$ is the socle of $S(\FF P)$.

By~(\ref{R}), we have $\Im(R|_{S(\FF P^0)})\cong S(\FF
P^0)/\Ker(R|_{S(\FF P^0)})$. Since any nonzero submodule of $S(\FF
P)$ has socle $X$, the quotient $S(\FF P^0)/\Ker(R|_{S(\FF P^0)})$
also has socle $X$. Inspecting the structure of $S(\FF P^0)$ from
Figure~2 of~\cite{ST}, we see that the only quotient of $S(\FF P^0)$
having $X$ as the socle is $X-Z-\FF-W$. So, $\Im(R|_{S(\FF P^0)})$
is uniserial with socle series $X-Z-\FF-W$. It follows that the
self-dual module $U_1^\perp/U'_1$ has a submodule $Z-\FF-W$. Notice
that $U_1^\perp/U'_1$ has composition factors: $Z$ (twice), $\FF$
(twice), and $W$, it must have the structure $Z-\FF-W-\FF-Z$. We now
can conclude that the structure of $S(\FF P)$ is $X-Z-\FF-W-\FF-Z-X$
by Proposition~\ref{mainUodd}.

\medskip

(iv) $\ell=3; n\equiv1\pmod3$: Then we have
$s=(2^{4n-3}-2^{2n-1})/3=0$. Therefore by Lemma~\ref{liebeck},
$\langle U_1,U_1\rangle=0$ or equivalently $U_1\subseteq U_1^\perp$.
In the proof of Lemma~\ref{lemmaUodd3}, we saw that $U_1=T(\FF
P)\oplus U'_1\cong \FF\oplus X$. Also, by the self-duality of $U_1$,
$\FF P/U_1^\perp\cong U_1$. It follows that, by
Lemmas~\ref{lemmaUodd1} and~\ref{lemmaUodd3}, $U_1^\perp/U_1$ has
composition factors: $Z$ (twice) and $W$.

We will show that $U_1^\perp/U_1$ is actually uniserial with socle
series $Z-W-Z$. Suppose not, so that $W$ is a submodule of
$U_1^\perp/U_1$. The self-duality of $U_1^\perp/U_1$ then implies
that $W$ is a direct summand of $U_1^\perp/U_1$. By
Proposition~\ref{mainUodd}, any nonzero submodule of $\FF P$ which
is not $T(\FF P)$ has $X=U'_1$ in the socle. Therefore,
by~(\ref{R}), $\FF P^0/\Ker R\cong \Im(R)$ has the socle containing
$X$. Inspecting the submodule lattice of $\FF P^0$ given in Figure~2
of~\cite{ST}, we conclude that $\Im(R)$ has the socle series
$X-Z-(\FF\oplus W)$. This shows that $W$ cannot be a direct summand
in $U_1^\perp/U_1$, a contradiction.

We have shown that $U_1^\perp/U_1$ is uniserial with socle series
$Z-W-Z$. Now we temporarily set $\FF_1:=T(\FF P)$ and $\FF_2:=\FF
P/S(\FF P)\cong U_1^{'\perp}/U^\perp_1$. From the previous
paragraph, $\FF P$ has a submodule $\Im(R)$ with socle series
$U_1'-Z-(\FF\oplus W)$ and the fact that $U_1=\FF_1\oplus U'_1$,
$\Im(R)$ must be $X-Z-(\FF_2\oplus W)$. Arguing similarly, $\FF P$
has a quotient having the socle series $(\FF_1\oplus W)-Z-X$. The
structure of $\FF P$ is now determined as given in
Table~\ref{tableUodd}.

\medskip

(v) $\ell=3; n\equiv2\pmod3$: Then $|P|=(2^{4n+1}+2^{2n})/3\neq 0$
and therefore $\FF P=T(\FF P)\oplus S(\FF P)$. It follows that
$S(\FF P)$ has composition factors: $\FF$, $X$ (twice), $Z$ (twice),
and $W$ by Lemmas~\ref{lemmaUodd1} and~\ref{lemmaUodd3}.

Proposition~\ref{mainUodd} implies that $X$ is the socle of $S(\FF
P)$. By Lemma~\ref{main3} and the self-duality of $S(\FF P)$, the
top layer of the socle series of $S(\FF P)$ is (isomorphic to) $X$
and hence any nontrivial quotient of $S(\FF P)$ has the top layer
$X$. From~(\ref{Q}) and Lemma~\ref{main1}, we have $\Im(Q|_{S(\FF
P)})\cong S(\FF P)/\Ker(Q|_{S(\FF P)})$. It follows that $X$ is the
top layer of the socle series of $\Im(Q|_{S(\FF P)})\subseteq S(\FF
P^0)$. Inspecting the submodule lattice of $\FF P^0$ given in
Figure~$2$ of~\cite{ST}, we see that $\Im(Q|_{S(\FF P)})$ must be
$(\FF\oplus W)-Z-X$. In other words, $S(\FF P)$ has a quotient
isomorphic to $(\FF\oplus W)-Z-X$.

By self-duality of $X=U'_1$ and $S(\FF P)$, $U_1^{\perp}/U'_1$ is
self-dual and has composition factors: $Z$ (twice), $\FF$, and $W$.
If either $\FF$ or $W$ is a submodule of $U_1^{\perp}/U'_1$, it
would be a direct summand in $U_1^{\perp}/U'_1$, which contradicts
the conclusion of the previous paragraph. So $Z$ must be the socle
of $U_1^{\perp}/U'_1$ and therefore the socle series of
$U_1^{\perp}/U'_1$ is $Z-(\FF\oplus W)-Z$. Finally, we obtain the
socle series of $S(\FF P)$: $X-Z-(\FF\oplus W)-Z-X$, as described.
\hfill$\Box$

\medskip

\textbf{Acknowledgement}: The authors are grateful for useful
comments and suggestions from Ulrich Meierfrankenfeld and Pham Huu
Tiep. The second author thanks the Department of Mathematics at
Michigan State University for the supportive and hospitable work
environment.


\end{document}